\documentclass[11pt,reqno]{amsart}

\usepackage[colorlinks=true,citecolor=black!60!green,linkcolor=black!60!red,filecolor=black!60!cyan,urlcolor=black!60!magenta]{hyperref}
\usepackage[text={6.1in,9in},centering]{geometry}
\usepackage{amssymb,amsmath,amsthm,mathtools,extpfeil,mathscinet}
\usepackage[all,cmtip]{xy}
\usepackage{tikz}
\usepackage{tikz-cd}
\usepackage[shortlabels]{enumitem}
\usepackage{xcolor}

\makeatletter
\newtheorem*{rep@theorem}{\rep@title}
\newcommand{\newreptheorem}[2]{%
\newenvironment{rep#1}[1]{%
 \def\rep@title{#2 \ref{##1}}%
 \begin{rep@theorem}}%
 {\end{rep@theorem}}}
\makeatother

\newtheorem{thm}{Theorem}
\newreptheorem{thm}{Theorem}
\newtheorem*{thm*}{Theorem}
\newtheorem*{lem*}{Lemma}
\newtheorem{thmA}{Theorem}

\newtheorem{lem}[thm]{Lemma}
\newtheorem{prop}[thm]{Proposition}
\newtheorem*{prop*}{Proposition}
\newtheorem{cor}[thm]{Corollary}
\newreptheorem{cor}{Corollary}
\newtheorem{conj}[thm]{Conjecture}
\theoremstyle{definition}
\newtheorem*{setting}{Setting and notation}
\newtheorem{defn}[thm]{Definition}
\newtheorem{rmk}[thm]{Remark}

\newtheorem{ex}[thm]{Example}

\numberwithin{thm}{section}
\numberwithin{equation}{section}

\DeclareMathOperator{\id}{id}
\DeclareMathOperator{\img}{im}
\DeclareMathOperator{\vol}{vol}

\DeclareMathOperator{\link}{lk}
\DeclareMathOperator{\Hom}{Hom}

\DeclareMathOperator{\Prim}{Prim}
\newcommand{\ph}{\varphi}
\newcommand{\epsi}{\varepsilon}

\newcommand*\strong[1]{\textbf{\textit{#1}}}

\newcommand{\QQ}{\mathbb{Q}}
\newcommand{\NN}{\mathbb{N}}
\newcommand{\ZZ}{\mathbb{Z}}
\newcommand{\RR}{\mathbb{R}}
\newcommand{\mubar}{\bar\mu}

\begin{document}
\title{Milnor invariants and thickness of spherical links}
\author{Rafa{\l} Komendarczyk}
\author{Robin Koytcheff}
\author{Fedor Manin}

\begin{abstract}
  The ropelength of a knot or link is the minimal number of inches of 1-inch-thick rope that it takes to tie it.  The relationship of this measurement to knot and link invariants has been studied by various authors.  We give the first results of this type for higher-dimensional spherical links, generalizing work of the first author and Michaelides in the classical case.  We find optimal asymptotic bounds on their Milnor invariants in terms of thickness, uncovering a dichotomy between a polynomial and an exponential regime.  Along the way, we give a detailed treatment of these Milnor invariants and their properties using Massey products.  As an application, we resolve a question of Freedman and Krushkal.
\end{abstract}
\maketitle

\setcounter{tocdepth}{1}
\tableofcontents

\section{Introduction}

\subsection{Background}
The goal of quantitative topology is to understand the geometric ``complexity'', in any relevant sense, of the objects whose existence is predicted by algebraic and geometric topology.  When studying smooth embeddings of a manifold $M$ in a compact ambient manifold $N$, a natural geometric measure of complexity is the \strong{thickness}, also called the \strong{reach} or (more descriptively) the \strong{normal injectivity radius}: the largest $\tau$ such that the exponential map on the $\tau$-neighborhood of the zero section in the normal bundle of $M$ is still an embedding.  Locally, this bounds curvature, and globally, how close distant points on $M$ can come to touching.  The more complicated the embedding, the smaller its thickness must be.  Variations on this notion exist for simplicial complexes and their PL embeddings.

With this setup, several natural questions come to mind:
\begin{enumerate}
    \item How does the least complexity of an embedding of $M$ depend on the topology, or some intrinsic measure of the complexity, of $M$?
    \item For a fixed $M$, how does the minimal complexity of an embedding in an isotopy class relate to the topological invariants of that isotopy class?
    \item Within a fixed isotopy class, what is the least complexity of an isotopy between two embeddings as a function of the complexity of the embeddings?
\end{enumerate}
All of these have been investigated in certain contexts.  (Topologically, embedding theory splits into several domains in which different tools and techniques can be used, and these suggest different approaches to the quantitative problems.)

The first question has been studied in high codimension: in the Whitney embedding range, in which generic maps are embeddings, Gromov and Guth \cite{GrGu} showed that one can find such embeddings with low complexity\footnote{Portnoy \cite{Portnoy} has recently improved on their result, obtaining the best possible asymptotic growth.}; Freedman and Krushkal \cite{FK} showed that the complexity can get much greater as soon as one leaves the Whitney range.

The third question, in contrast, has been studied in codimensions 1 \cite{NabSph} and 2 \cite{NWknots}; in both cases, the answer is that this function grows unimaginably fast.

Finally, the second question has been studied, starting in the 1990s \cite{CKS2,CKS}, mainly in the classical context of knot theory in $\RR^3$.  Here, the complexity is measured not by thickness but by a closely related measurement called \strong{ropelength}: informally, the number of inches of 1-inch-thick rope needed to tie a given knot or link.  Past work has investigated the relationships of ropelength to other notions of complexity of knots and links, such as crossing number \cite{BS,DEPZ}, as well as to values of algebraic invariants, such as Milnor invariants of links \cite{KM}.

The present paper appears to be the first attempt to study the second question in higher dimensions, and perhaps the first paper to study quantitative embedding theory outside very low and very high codimensions.  In this range, where $\dim M$ is between roughly $\frac{2}{3}\dim N$ and $\dim N-3$, 
%\noteblue{(Robin: why not start around $\frac{1}{2}\dim N$, where the Whitney range ends?)}
%Fedya says between 1/2 dim N and 2/3 dim N, one doesn't have to start with immersions, and starting with immersions is annoying.
the state-of-the-art topological tool for studying embeddings is the manifold calculus of Goodwillie, Klein, and Weiss 
\cite{Weiss:1999, GW:1999, GK:2015, GKW:2001, GKW-Haefliger}.  As the general setting is quite complicated, we do not (yet) attempt to study it quantitatively, but stick to embeddings of simple manifolds, namely disjoint unions of spheres, possibly of different dimensions.  We study higher-dimensional generalizations of Milnor invariants of these higher-dimensional links, generalizing the work of the first author and Michaelides \cite{KM} for classical links.

\subsection{Milnor invariants}
Milnor invariants of links, introduced by Milnor in his PhD thesis \cite{Milnor,Milnor2}, are a generalization of the linking number which detects higher-order linking; for example, the unique order-3 Milnor invariant detects the linking of the Borromean rings, and more generally, \strong{Brunnian links} (those that become unlinked when any component is removed) have Milnor invariants that detect their linking.  Milnor studied these invariants for classical links, i.e., embeddings $S^1 \sqcup \dots \sqcup S^1 \to S^3$.

In general, each sequence of $k$ link components yields a Milnor invariant $\mubar(i_1,\ldots,i_k)$ associated to that sequence.  These Milnor invariants are well-defined integers if the invariants associated to subsequences are zero; otherwise they are defined modulo an integer.

Milnor invariants come in two flavors.  Those studied in \cite{Milnor} 
%(sometimes referred to as \strong{Milnor linking numbers}) 
are indexed by distinct components of the link and are invariant under \strong{link homotopy}: a deformation which keeps components disjoint but (unlike an isotopy) need not be injective on each component.  Those studied in \cite{Milnor2} are more general and detect more links: for example, $\mubar(1,1,2,2)$ is nontrivial for the link-homotopically trivial Whitehead link.

Milnor originally defined his invariants by studying lower central series quotients of the fundamental group of the link complement.  Turaev \cite{Turaev} and later Porter \cite{Porter} gave an alternate definition based on Massey products in cohomology.  This cohomological definition has the advantage of extending verbatim to higher dimensions, and this is the definition of Milnor invariants we use in this paper.  In part to justify this choice, we show that versions of the relations between Milnor invariants proved originally by Milnor hold in higher dimensions (see \S\ref{S:Milnor-props}), and demonstrate that, for invariants with distinct indices, our definition coincides (at least for Brunnian links, but conjecturally in a wider range of cases where both are defined) with another higher-dimensional generalization of Milnor invariants defined by Koschorke \cite{Kos1,Kos2} (see \S\ref{S:Massey-Milnor-Koschorke}).  Koschorke's invariants have the advantage that they can be defined even when the individual link components are not embedded; this is not relevant to our results since they depend on the thickness of individual components.  They are also invariant with respect to a relation which is coarser than link homotopy in some cases.

The classification of links is much simplified when all components have codimension at least $3$ \cite{HaefLinks,CFS}; here, equivalent Milnor invariants were previously defined by Haefliger and Steer \cite{HaeSt} for triple links and Turaev \cite{Turaev2} in general.  In this case, links $S^{p_1} \sqcup \cdots \sqcup S^{p_r} \to S^m$ form a group $L_{(p_1,\ldots,p_r)}^m$, under the (well-defined) operation of connected sum, and all nontrivial Milnor invariants are well-defined homomorphisms of this group to $\ZZ$.  Moreover, comparing with the classification in \cite{CFS} one sees that, together with the Haefliger knotting invariants of the individual components, the Milnor invariants distinguish rational isotopy classes, i.e., elements of $L_{(p_1,\ldots,p_r)}^m \otimes \QQ$.

\begin{setting}
In the most general case, we consider $C^\infty$-embeddings $S^{p_1} \sqcup \cdots \sqcup S^{p_r} \to S^m$ with $1 \leq p_i \leq m-2$, where $S^m$ is the sphere of unit radius in $\RR^{m+1}$.  The Milnor invariant $\mubar(i_1,\ldots,i_d)$, $d\geq 2$, is a well-defined integer whenever certain lower-order invariants vanish\footnote{We focus on rational phenomena, but in general, $\mubar(i_1,\ldots,i_d)$ would be well-defined modulo the gcd of lower-order invariants.}, and it is always zero unless
\begin{equation}
\label{eq:dim-assumption}
    p_{i_1}+\cdots+p_{i_d}=(m-2)(d-1)+1.
\end{equation}  
Throughout the paper, we write $q_i=m-p_i-1$, and the dimensional assumption \eqref{eq:dim-assumption} can be rewritten perhaps more transparently as
\begin{equation}
\label{eq:dim-assumption-reformulation}
q_{i_1}+\cdots+q_{i_d}-(d-2)=m-1.
\end{equation}
\end{setting}

\subsection{Main results and open questions}
The strongest result in this paper achieves a more or less complete understanding of the relationship between thickness and Milnor invariants with distinct indices.  Without loss of generality, the number of components $r$ is the depth $d$ of the invariant, and our result is as follows:
\begin{thmA} \label{main}
    Let $f:S^{p_1} \sqcup \cdots S^{p_d} \to S^m$ be a link of thickness $\tau$ satisfying
    \begin{equation}
    \label{eq:dim-assumption-nr}
        p_1+\cdots+p_d=(m-2)(d-1)+1,
    \end{equation}
    such that Milnor invariants indexed by proper subsequences of $(1,\ldots,d)$ are trivial. 
    \begin{enumerate}[(i)]
    \item \label{intro-case:poly} If at least one of the $p_i$ is $1$, then
    \[\mubar(1,\ldots,d) \leq C(m,d)\tau^{-(m+1)(d-1)}.\]
    \item \label{intro-case:link} If $d=2$, then $\mubar(1,2)$ is the linking number of the two components, and is bounded by $C(m)\tau^{-(m+1)}$.
    \item \label{intro-case:exp} In all other cases, $\mubar(1,\ldots,d) \leq \exp(C(m,d)\tau^{-m})$.
    \end{enumerate}
    Here, $C(m)$ and $C(m,d)$ are constants.  The estimates are asymptotically optimal for every combination of $m$, $d$, and $p_1,\ldots,p_d$.
\end{thmA}
In \S\ref{S:examples} we present families of examples that realize this asymptotic growth.

The bound on the linking number in part \ref{intro-case:link}, and its sharpness, is easy and was observed at the latest by Freedman and Krushkal \cite{FK}, though it is closely related to earlier ideas of Gromov \cite{GroHED}.  In the classical setting $m=3$, Theorem \ref{main}\ref{intro-case:poly} recovers the result of the first author and Michaelides \cite{KM}, albeit without their explicit estimates on the constant $C(3,d)$.  In all other cases, the results are new; moreover, the examples we use to show that they are asymptotically sharp are new even in the classical setting.

Perhaps the most surprising aspect is the sharpness of part \ref{intro-case:exp}: that Milnor invariants can be exponentially large in terms of the thickness.  We give a method of constructing examples which works whenever there are at least two components of codimension at least $3$; by \eqref{eq:dim-assumption} this is equivalent to having no $1$-dimensional components.  It relies on the existence of metrics on the sphere which contain long isometrically embedded telescopes, and therefore pairs of small lower-dimensional spheres with large linking number.  For antecedents to this technique, see for example \cite[\S4]{FK} and \cite[Appendix A]{Has}.

This family of examples also demonstrates that thickness, a measure of the complexity of an \emph{embedded object}, can behave very differently from measures of complexity of embeddings \emph{qua maps}.  To illustrate this, we show in Theorem \ref{thm:upper-bound}\ref{case:Lip} that if in addition to thickness we bound the local bilipschitz constant of the embedding $f$, we recover a polynomial bound on Milnor invariants.

We achieve somewhat less for Milnor invariants with repeated indices, leaving some tantalizing open questions:
\begin{thmA} \label{main-repeated}
    Let $f:S^{p_1} \sqcup \cdots \sqcup S^{p_r} \to S^m$ be a link of thickness $\tau$ such that \eqref{eq:dim-assumption} holds and all Milnor invariants indexed by proper cyclic subsequences of $(\ell_1,\ldots,\ell_d)$ are trivial.
    \begin{enumerate}[(i)]
%    \item \label{intro-case:classical-rep} If $m=3$ (and therefore all the $p_{\ell_i}$ are $1$), then
%    \[\mubar(\ell_1,\ldots,\ell_d) \leq C(d)\tau^{-4(d-1)}.\]
    \item \label{intro-case:poly-rep} If one of the $p_{\ell_i}$ is $1$ (in which case the rest must be $m-2$), then
    \[\mubar(\ell_1,\ldots,\ell_d) \leq C(m,d)\tau^{-2(m+1)(d-1)}.\]
    \item \label{intro-case:exp-rep} In all other cases, $\mubar(\ell_1,\ldots,\ell_d) \leq \exp(C(m,d)\tau^{-m})$.
    \end{enumerate}
    Again, 
    %$C(d)$ and 
    $C(m,d)$ is a constant.
\end{thmA}
In the polynomial regime \ref{intro-case:poly-rep}, we get a different estimate from that of Theorem \ref{main}, which may not be sharp.  When $m \geq 4$, our construction of links with large homotopy invariants can easily be modified to produce examples with
\[\mubar(\ell_1,\ldots,\ell_d)=C(m,d)\tau^{-(m+1)(d-1)},\]
but it is unclear how---and in what direction!---one can bridge the quadratic gap.

Even in the classical setting $m=3$, the results are new, not covered by those of \cite{KM}.   It would be interesting to find examples of thick links with large Milnor invariants, even for the simplest nontrivial case $\mubar(1,1,2,2)$ (the invariant that demonstrates the nontriviality of the Whitehead link).

In the exponential regime \ref{intro-case:exp-rep}, we believe that the estimate is still sharp, but constructing examples presents other challenges.  First, there is some complicated numerology, depending on the parity of codimensions, involved in deciding which collections of indices give nontrivial invariants.  Secondly, one may hope to create new examples by taking the examples we constructed with large link homotopy invariants and taking connected sums of some components.  But an identity between Milnor invariants of the original link and the connected sum may only hold with some error term which must also be estimated.

In another direction, there is a connection between links with nontrivial Milnor invariants and Haefliger's construction of a family, indexed by the integers, of smoothly knotted, PL unknotted spheres \cite{Haef}.  As a next project in this direction, one would hope to harness this connection to give an estimate on Haefliger's invariant in terms of thickness.

For the final result of this paper, we apply our results in the case $m=4$ to answer a question of Freedman and Krushkal \cite[\S5]{FK}.  They constructed simplicial $n$-complexes of uniformly bounded local combinatorial complexity which embed in $[0,1]^{2n}$, but such that the thickness of any embedding is exponentially small in the number of simplices.  The proof of this fact relied on the fact a certain pair of spherical subcomplexes are always linked in any embedding.  In dimension $4$, they also constructed analogous complexes $\bar K_{q,l}$ using $q$th-order linking, and asked whether the thickness of any embedding of these complexes is likewise exponentially small in $l$.  We give a positive answer in \S\ref{S:FK}:
\begin{thmA}\label{thm-C}
    Any embedding of the $2$-complex $\bar{K}_{q,l}$ (which has $O(q+l)$ simplices and uniformly bounded local combinatorial complexity) has thickness at most $c^{-l}$, where the constant $c>1$ depends on $q$. %\noteblue{Moreover, the refinement complexity of  $\{\bar{K}_{q,l}\}$ satisfies:
%\[
%C^{m_l} < \operatorname{rc}(\bar{K}_{q,l}, 4) < \infty.
%\]
%}Here the constants $c, C > 1$ depend on $q$ only. 
\end{thmA}

\subsection{Structure of the paper}
In \S\ref{S:examples}, we construct examples to prove that the bounds in Theorem \ref{main} are asymptotically optimal.  Although some facts from later sections are used, these examples are the most readily obtained of our main results, and \S\ref{S:examples} is otherwise independent of the rest of the paper.  
Next, \S\ref{S:periods}, \S\ref{S:Milnor-and-Massey}, and \S\ref{S:repeated} provide the foundation for the proofs of our main results, together with some related auxiliary content.  
In \S\ref{S:periods}, we review homotopy periods in terms of Sullivan's minimal models, and we establish formulas evaluating homotopy periods on Whitehead products in wedges of spheres.  
Integral formulas for Milnor invariants with distinct indices are developed via Massey products in \S\ref{S:Milnor-and-Massey}, where we also establish their properties and explore their relation to Koschorke's link homotopy invariants.
Subsequently, we generalize this treatment to Milnor invariants with repeated indices in \S\ref{S:repeated}.  
There we also discuss the simplest such invariants and note how Milnor invariants determine rational isotopy classes of links modulo knotting.
The proofs of the upper bounds in Theorems \ref{main} and \ref{main-repeated} are the subject of \S\ref{S:upper-bounds}, which begins with a development of coisoperimetric inequalities.  
In \S\ref{S:FK}, we prove Theorem \ref{thm-C} on exponentially thin 2-complexes in dimension 4.
%which answers a question of Freedman and Krushkal.

\subsection{Acknowledgments}
We thank Nir Gadish for helpful discussions related to homotopy periods.  
The second author was partially supported by the Louisiana Board of Regents Support Fund, contract number LEQSF(2019-22)-RD-A-22 and NSF grant DMS-2405370.  
The third author would like to thank the Hausdorff Institute\footnote{Funded by the Deutsche Forschungsgemeinschaft (DFG, German Research Foundation) under Germany's Excellence Strategy -- EXC-2047/1 -- 390685813.}, where considerable parts of the paper were written; he was also partially supported by NSF grant DMS-2204001 and an NSERC Discovery Grant.

\section{Thick Brunnian links with large Milnor invariants} \label{S:examples}

In this section we give examples of thick Brunnian links that demonstrate that the upper bounds of Theorem \ref{main} give the right rate of growth in all cases.  None of the later sections depend on this one, and we will use results from \S\ref{S:periods} and \S\ref{S:Milnor-and-Massey} to verify the attainment of these upper bounds.

We give two types of geometric constructions which are similar topologically; we start by describing these topological commonalities.  Suppose that $p_1,\ldots,p_d$ satisfy the dimensional assumption \eqref{eq:dim-assumption-nr}.  We take $d-1$ spheres $S^{p_1} \sqcup \cdots \sqcup S^{p_{d-1}}$ which form a trivial link in $S^m$.  Their complement is homotopy equivalent to
\[S^{q_1} \vee \cdots \vee S^{q_{d-1}},\]
where $q_i = m-p_i-1$ as previously.  We now explain how to embed a $p_d$-sphere in this complement so as to produce a large Milnor invariant.

Recall that the \strong{Whitehead product} $[\alpha,\beta]$ of two elements $\alpha \in \pi_m(X)$ and $\beta \in \pi_n(X)$ is the element of $\pi_{m+n-1}(X)$ given by the composition
\[S^{m+n-1} \xrightarrow{\text{attaching map of the top cell of }S^m \times S^n} S^m \vee S^n \xrightarrow{\alpha \vee \beta} X.\]
In particular, the Whitehead product of two elements of $\pi_1$ is their commutator; if $m=1$ and $n>1$, then $[\alpha,\beta]=\alpha \cdot \beta-\beta$, where $\cdot$ represents the action of $\pi_1$ on higher homotopy groups.  It is graded commutative, and for $*>0$, the Whitehead product on $\pi_{*+1}$ satisfies the axioms of a graded Lie algebra operation, further justifying the notation $[{\cdot},{\cdot}]$.

When there is no $1$-dimensional component, the homotopy class of our embedding will be a large multiple of the iterated Whitehead product
\[[e_1,[e_2, \cdots [e_{d-2},e_{d-1}] \cdots ]],\]
where $e_i$ corresponds to a meridian of component $i$, meaning a map of $S^{q_i}$ that has linking number 1 with $S^{p_i}$.  
(Indeed, condition \eqref{eq:dim-assumption-nr} implies that component $d$ has the same dimension as this Whitehead product, i.e., $p_d = q_1 + \dots + q_{d-1} - (d-2)$.)  The resulting link is Brunnian with a correspondingly large Milnor invariant.

In the case where $p_d=1$ (so that the formula above gives an iterated commutator), the homotopy class will not literally be a large multiple of this iterated commutator, but it will be equivalent to such a multiple from the point of view of link homotopy, as explained below.

\subsection{Polynomial construction} \label{S:poly}

We start by proving the sharpness of the upper bound in Theorem \ref{main}\ref{intro-case:poly}, where $p_1=\cdots=p_{d-1}=m-2$ and $p_d=1$.

\begin{thm}\label{thm:polynomial-thickness-mu}
  There is a constant $\epsi=\epsi(m,d)>0$ such that for any
  $n \in \NN$, there is a $d$-component link 
  $S^{m-2} \sqcup \dots \sqcup S^{m-2} \sqcup S^1 \to S^m$
  with thickness $\epsi n^{-1}$ and
  \[\mubar(1,\ldots,d)=n^{(m+1)(d-1)}.\]
\end{thm}
\begin{proof}
  The complement of a set of $d-1$ unlinked $(m-2)$-spheres is homotopy equivalent to a wedge of circles; in this complement, we will embed a circle which traces out an element of $\pi_1 \cong F_{d-1}$ of the form
  \[g=[e_1^{n^{m+1}},[\cdots[e_{d-2}^{n^{m+1}},e_{d-1}^{n^{m+1}}]\cdots]],\]
  where $e_i$ is the inclusion of the $i$th circle in the wedge.  We first show that this implies the Milnor invariant is as desired.
  
  Suppose first that $m=3$; we explain the general case later.  As explained in \cite{Milnor}, the Milnor invariants (and indeed the link homotopy type) of the resulting link are determined by the element induced by $g$ in the \strong{Milnor group}, a certain quotient of the fundamental group of its complement.  Specifically, the Milnor group of the trivial link is the \strong{free Milnor group}
  \[\langle e_1,\ldots,e_{d-1} \mid [e_i,w^{-1}e_iw], 1 \leq i \leq d-1, w\text{ any word}\rangle.\]
  In this group, we have
  \[[e_i^a,w] = [e_i,w]^{e_i^{a-1}}[e_i^{a-1},w]=[e_i,w][e_i^{a-1},w] = \cdots = [e_i,w]^a,\]
%   \noteblue{r: I had to use $[e_i,we^{-1}_iw^{-1}]=1$ in the second step, is it obvious? also I think we need $[e_i,w^a]=[e_1,w]^a \mod G_3$, the following formula from \cite[p. 290]{Magnus-Karrass-Solitar:2004} should suffice:
% \[
%  [a,bc]=[a,c][a,b][[a,b],c]
% \]  
% (the convention for $[a,b]$ is $a^{-1}b^{-1}ab$.)
% }
  so the element $g$ is equivalent to
  \[[e_1,[\cdots[e_{d-2},e_{d-1}]\cdots]]^{n^{(m+1)(d-1)}}.\]
  Therefore the Milnor invariant $\mubar(1,\ldots,d)$ of the resulting link is $n^{(m+1)(d-1)}$ \cite[p.~190]{Milnor}.

  For higher $m$, the computation reduces to that in dimension $3$: we can choose a $3$-sphere $P \subset S^m$ that intersects the $(m-2)$-spheres in unlinked circles, and our map from the circle can be homotoped into $P$.  Proposition \ref{Milnor=period} will allow us to compute the Milnor invariant $\mubar(1,\ldots,d)$ as a \emph{homotopy period} of the $d$th component in the complement of the others, that is, by applying a sequence of primitives and wedge products to the $1$-forms Alexander dual to the components.  Restricting this $m$-dimensional computation to $P$ yields precisely the $3$-dimensional computation.

  Now we give the construction of a thick link.  Throughout the construction, $\epsi_0, \dots, \epsi_4$ will be positive numbers that may depend  on $m$ and $d$, but not on $n$.  First, fix $(m-2)$-spheres $S_1,\ldots,S_{d-1}\subset S^m$, circles $C_1,\ldots,C_{d-1} \subset S^m$, and $\epsi_0>0$ such that both $S_1 \sqcup \cdots \sqcup S_{d-1}$ and  $C_1 \sqcup \cdots \sqcup C_{d-1}$ have thickness $2\epsi_0$, and such that $S_i$ has linking number $1$ with $C_i$.  In the $\epsi_0$-neighborhood $N_i$ of $S_i$ we can place $n^2$ copies of $S^{m-2}$ so that their disjoint union is $\epsi_1 n^{-1}$-thick (i.e., so that its cross-sectional volume is proportional to $n^{-2}$).  
  %(This fact can be seen, at least heuristically, by dividing the unit cube $[0,1]^m$ into $n^m$ subcubes congruent to $[0,n^{-1}]^m$.)  
  Finally, we take the connected sum of all these spheres, via tubes of controlled thickness between neighboring spheres, to create a single $(m-2)$-sphere $S_i'$ in $N_i$ whose projection to $S_i$ has degree $n^2$ and which has thickness $\epsi_2 n^{-1}$.  By the same token, we can embed a circle $C_i'$ in the $\epsi_0$-neighborhood of $C_i$ which winds $n^{m-1}$ times around $C_i$ with thickness $\epsi_3 n^{-1}$.  Then the linking number of $C_i'$ and $S_i'$ is $n^{m+1}$.  By connecting together copies of the various $C_i'$, we get a loop representing the word $g$, though it may not a priori be an embedding.  However, it travels along any $C_i'$ (in either direction) at most $2^{d-1}$ times, so we can perturb it to an $\epsi_4 n^{-1}$-thickly embedded loop.  Putting $\epsi=\epsi_4$ completes the proof.
\end{proof}

\begin{figure} %[htbp]
    \centering
    \includegraphics[width=0.9\linewidth]{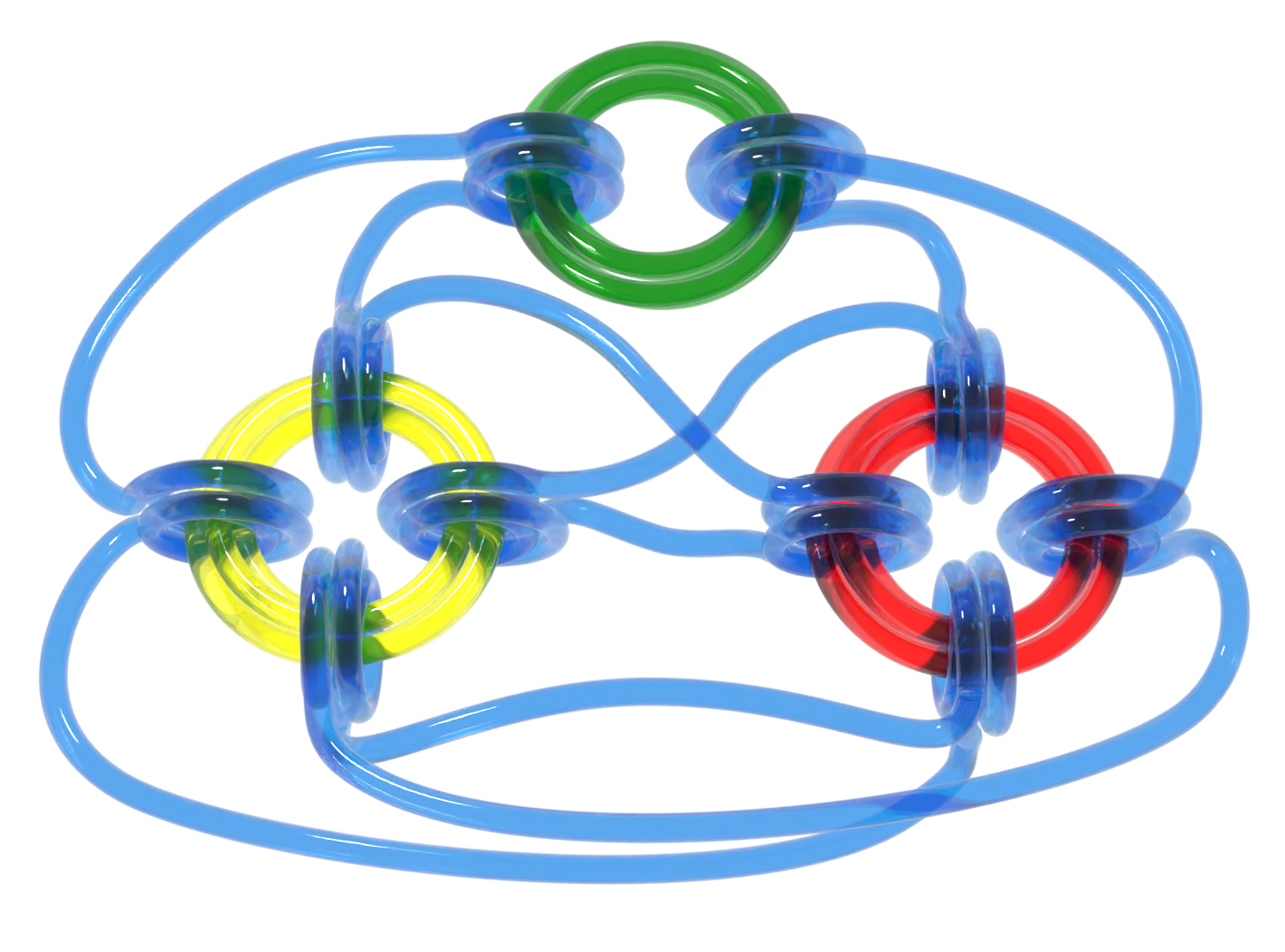}
    \caption{Illustration of the $4$--component link obtained in Theorem \ref{thm:polynomial-thickness-mu} for $d=4$, $n=2$, $m=3$, with thickness $\tau=\varepsilon n^{-1}$ and $\mubar(1,2,3,4)=n^{12}$. Components $e_1$, $e_2$ and $e_3$ are respectively in green, red and yellow, the commutator component $g=[e^{2^4}_1,[e^{2^4}_2,e^{2^4}_3]]=[e_1,[e_2,e_3]]^{2^{12}}$ is shown in blue. }
    \label{fig:commutator-link}
\end{figure}

\subsection{Exponential construction} \label{S:exp}
Our next task is to prove the sharpness of the upper bound in Theorem \ref{main}\ref{intro-case:exp}.
We start with the special case that involves the fewest components of the lowest dimensions possible: for each $n \in \NN$, we will construct a link $S^2 \sqcup S^2 \sqcup S^3 \subset S^5$ of thickness $\epsi n^{-1}$ and triple linking number $2^{n^5}$, where $\epsi>0$ is some constant.  

We start by choosing embeddings of the two copies of $S^2$: we embed them both as round spheres of radius $n^{-1}$.
We will describe an embedding of $S^3$ of thickness\footnote{We use the notation $\sim f(n)$ to mean ``proportional to $f(n)$'', i.e., a constant times $f(n)$.}
$\sim n^{-1}$ which represents $2^{n^5}[e_1,e_2]$.  Before doing so, we discuss how its homotopy class is detected.  Given a map $f:S^3 \to S^2 \vee S^2$ that is smooth away from the wedge-point, we can take the linking number of preimages of generic points in the two copies of $S^2$.  As with Hopf's original formulation of the Hopf invariant, one sees that this linking number is homotopy invariant and additive; in other words, it defines a homomorphism $\pi_3(S^2 \vee S^2) \to \ZZ$.\footnote{In \S\ref{S:periods}, we explain Sullivan's recipe for computing such homomorphisms using differential forms, generalizing Whitehead's formulation of the Hopf invariant.  This may be the easiest way of making the informal description in this section precise.}  Moreover, this homomorphism is surjective: the Whitehead product of the identity maps on the two spheres has an explicit representative which splits $S^3$ into two solid tori $S^1 \times D^2$ and maps each to the respective copy of $S^2$ via the $D^2$ coordinate; then the preimages of generic points are linked circles.  In our setting, the preimages of points are replaced by intersections of our embedded $S^3$ with ``Seifert surfaces'' of the two summands of $S^2$, i.e., the 3-balls bounded by these 2-spheres.  Thus we will construct an embedding of $S^3$ in which the linking number of these intersections is exponential.

We build this embedding out of two types of standard building blocks.  The first building block, which we call $C$, is topologically a solid torus $S^1 \times D^2$.  We embed two copies of $C$ in $S^5$ by choosing copies of $S^2$ which link with the two embedded copies of $S^2$, cutting out a small disk from each, and taking the product of what remains with a small circle.

The second building block, which we call $B$, has the topology of the complement in $S^1 \times D^2$ of a tubular neighborhood of an embedded circle that winds twice around the $S^1$ factor.  
\begin{figure}%[h!]
    \centering
    \includegraphics[width=0.3\linewidth]{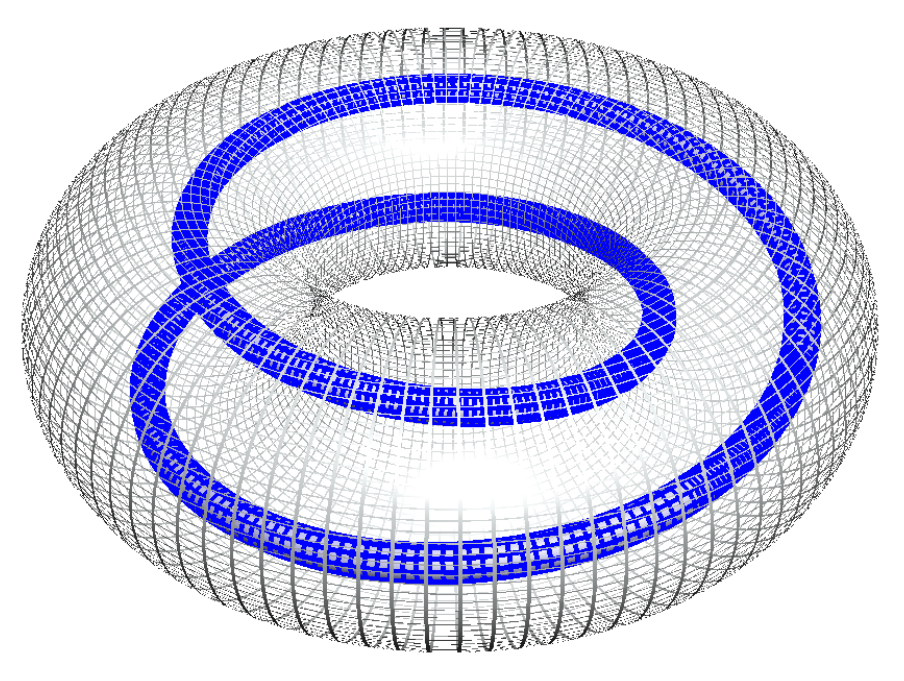}
    \caption{The building block $B$ is the complement in the solid torus of a neighborhood of the blue curve.}
    \label{fig:T12}
\end{figure}
We embed $B$ in a small ball in $\RR^5$ so that the outer torus and the inner torus have the same geometry.  (To see that this is possible, notice that it is possible even in $\RR^4$: thinking of the fourth dimension as time, we glue together the standard embedding of $B$ in $\RR^3$ with an isotopy that ``unwraps'' the inner torus to make it isotopic to the outer one.)

Now we build our embedding of $S^3$.  First decompose $S^3$ into two solid tori.  One of these is one copy of $C$.  The other will consist of $n^5$ nested copies of $B$ with the second copy of $C$ in the interior.  This nesting makes sense because the complement of $B$ in $S^1 \times D^2$ is another solid torus.  Note also that the interior copy of $C$ winds around the outer solid torus $2^{n^5}$ times.  Having already embedded the two copies of $C$, we place the copies of $B$ along a curve of thickness $n^{-1}$ which connects the two summands of $S^2$ and (with its tubular neighborhood) fills up the rest of $S^5$ along the way.  This curve has length $\sim n^4$ (since its neighborhood has cross-sectional volume $\sim n^{-4}$), so we can split it into $n^5$ sections of length $\sim n^{-1}$; each copy of $B$ is embedded in the normal tube of such a section with thickness $\sim n^{-1}$.

%%\noteblue{Robin: at any point in this argument, are we sloughing over details about the existence of embeddings that are treated more carefully in Lemma \ref{lem:Whitehead} or in the Theorem after it?}

Now the intersection of this embedding with a 3-ball filling an $S^2$ summand is a circular fiber inside the respective copy of $C$.  These two circular fibers have linking number $2^{n^5}$, so the embedding constructed has triple linking number $2^{n^5}$, as desired.

To prove the general result, we first show that any iterated Whitehead product can be represented by an embedding.
Recall that $q_i=m-p_i-1$.
%First we show that the Whitehead product itself can be embedded:
\begin{lem} \label{lem:Whitehead}
    Fix $p_1,\ldots,p_d$ such that 
    \[p:=q_1+\cdots+q_d-(d-1) \leq m-2,\]
    and consider an unlink $L=S^{p_1} \sqcup \cdots \sqcup S^{p_d} \subset S^m$.  Then there is an embedding representing
    \[[e_1, [\cdots [e_{d-1},e_{d}] \cdots ]] \in \pi_p(S^m \setminus L)\]
    whose image lies in a $(p+2)$-dimensional sphere in $S^m$, where 
    $e_i$ maps $S^{q_i}$ into the link complement as a meridian of the $i$th component $S^{p_i}$.  
\end{lem}
\begin{proof}
    We proceed by induction on $d$.  If $d=1$, we can just take a sphere linked with $S^{p_1}$ lying in a $(p+1)$-dimensional subspace.  Now suppose we have an embedding $f:S^{p-q_1+1} \hookrightarrow S^m \setminus L$ representing $\alpha=[e_2, [\cdots [e_{d-1},e_{d}] \cdots ]]$ whose image lies in a $(p-q_1+3)$-dimensional sphere in $S^m$.
      
    Recall that $[e_1,\alpha]$ is represented by a map as follows.  Write
    \[S^{p_1}=D^{q_1} \times S^{p-q_1} \cup S^{q_1-1} \times D^{p-q_1+1}.\]
    Then on each of the two solid tori, the map forgets the sphere coordinate and maps the disk coordinate to the wedge of spheres via $e_1$ and $\alpha$, respectively.  We would like to construct an embedding $g:S^p \hookrightarrow S^m \setminus L$ which is homotopic to this map.
      
    The normal bundle of $f(S^{p-q_1+1}) \subset S^{p-q_1+3}$ is determined by the Euler class, because the codimension is $2$.  The Euler class is the image of the Thom class \cite[Theorem 11.3]{MS} under the composition
    \[H^2(S^m, S^m \setminus f(S^n)) \to H^2(S^m) \to H^2(f(S^n)).\]
    This shows that this class is cohomologically trivial.  So we can choose a framing (unlinked with the image of $f$, if $p-q_1+1=1$).  This extends to a framing of the $q_1$-dimensional normal bundle in $S^{p+1}$; by restricting this framing to a small sphere bundle, we get a map
    \[f':S^{q_1-1} \times S^{p-q_1+1} \to S^m \setminus L\]
    which lands in $S^{p+1}$ and, homotopically, represents a projection to the second coordinate followed by $\alpha$.  Similarly, we can get a map
    \[f'':S^{q_1} \times S^{p-q_1} \to S^m \setminus L\]
    which lands in $S^{p+1}$ and represents a projection to the first coordinate followed by $e_1$.

    Now by restricting $f'$ and $f''$ to preimages of the boundary of a small ball in $S^{p-q_1+1}$ and $S^{q_1}$, respectively, we get two embeddings of $S^{q_1-1} \times S^{p-q_1}$, each representing the normal sphere bundle of an unknotted sphere in $S^p$.  In $S^{p+1}$, these normal bundles are isotopic to the ``Clifford torus'', i.e., to the product of the standard embeddings $S^{q_1-1} \hookrightarrow \RR^{q_1}$ and $S^{p-q_1} \hookrightarrow \RR^{p-q_1+1}$, and hence to each other.  We create our map $g$ by connecting the punctured $f'$ and $f''$ via a movie of this isotopy in $S^{p+2}$.
\end{proof}

Now we establish the sharpness of the upper bound in Theorem \ref{main} \ref{intro-case:exp}:

\begin{thm}
  Suppose that $m-p_i>2$ for at least two values of $i$.  Then there is a
  constant $\epsi=\epsi(m,d)>0$ such that for any $n \in \NN$ there is a 
  link $S^{p_1} \sqcup \dots \sqcup S^{p_d} \to S^m$
  with thickness $\epsi n^{-1}$ and
  \[\mubar(1,\ldots,d)=2^{n^m}.\]
\end{thm}
\begin{proof}
  By Proposition \ref{cyclic}, Milnor invariants are symmetric under cyclic permutations of the indices, so we may assume without loss of generality that $p_1<m-2$ and that $p_i<m-2$ for some $1<i<d$.  Thus $q_1\geq 2$ and $q_i \geq 2$.  Recall that by condition \eqref{eq:dim-assumption-nr}, $p_d = q_1 + \dots + q_{d-1} -(d-2)$.
  
  The general construction is very similar to the special case of of $S^2 \sqcup S^2 \sqcup S^3 \to S^5$ described above.  We first embed $S^{p_1},\ldots,S^{p_{d-1}}$ as small spheres of radius $n^{-1}$.  We will next construct an embedding of $S^{p_d}$ in their complement which represents a multiple of an iterated Whitehead product, namely $2^{n^m}[e_{1}, [\cdots [e_{d-2},e_{d-1}] \cdots ]]$.  The Milnor invariant $\mubar(1,\ldots,d)$ of this link will then be $2^{n^m}$ by Proposition \ref{Milnor=period} and Proposition \ref{Whitehead/period}.
  
  As in the special case, we split $S^{p_d}$ into two solid tori: $T_1=S^{p_d-q_1} \times D^{q_1}$ and $T_2=D^{p_d-q_1+1} \times S^{q_1-1}$.

  By our dimensional assumptions, $q_1 \geq 2$ and $p_d-q_1 \geq 1$.  This
  allows us to choose a smaller solid torus $T'$ inside $T_1$ which is a tubular
  neighborhood of the connected sum of two $S^{p_d-q_1}$ fibers.  We let
  $B=T_1 \setminus T'$; as in the special case, we can embed $B$ as a cobordism between two
  isometric copies of $\partial T_1$.

  Once again, we further split $T_1$ into $n^m$ nested copies of $B$ capped off
  by a copy of $T_1$ that we call $C$.  We embed $S^{p_d}$ as follows, where for a product of spaces $X \times Y$, we write $\pi_X$ and $\pi_Y$ for the two projections:
  \begin{itemize}
  \item $T_2$ via an embedding which is a perturbation of the map
    \[[e_{2},[\cdots[e_{d-2},e_{d-1}]\cdots]] \circ \pi_{D^{p_d-q_1+1}}:D^{p_d-q_1+1} \times S^{q_1-1} \to S^m \setminus (S^{p_2} \sqcup \cdots \sqcup S^{p_{d-1}}),\]
    for example the restriction of the map $f'$ constructed in Lemma \ref{lem:Whitehead};
  \item $C$ via a perturbation of $e_{1} \circ \pi_{D^{q_1}}$; and
  \item the $n^m$ copies of $B$ along a curve of thickness $n^{-1}$ which fills
    up the rest of $S^m$, with one copy along each interval of length $n^{-1}$.
  \end{itemize}
  The embedding restricted to $T_1 \cap T_2$ lies in a small ball that does not intersect any of the other embedded spheres.  Shrinking this ball to a point, we get a map
  \[f_0:T_1/\partial T_1 \cong S^{p_d-q_1} \times S^{q_1} \to S^m \setminus S^{p_1}.\]
  The intersection of its image with a ``Seifert surface'' for $S^{p_1}$ is a
  $S^{p_d-q_1}$ fiber inside $C$; intersecting this fiber with an $S^{q_1}$ fiber in
  $T_1$ gives $2^{n^m}$ points.  Therefore the restriction of $f_0$ to an
  $S^{q_1}$ fiber is a map $S^{q_1} \to S^m \setminus S^{p_1}$ whose linking
  number with $S^{p_1}$ is $2^{n^m}$.  So the embedding of $S^{p_d}$ is homotopic
  to
  \[[2^{n^m}e_{1},[e_{2},[\cdots[e_{d-2},e_{d-1}]\cdots]]],\]
  and by the bilinearity of the Whitehead product on higher homotopy groups, we are done.
\end{proof}

\section{Homotopy periods} \label{S:periods}

%\noteblue{Should this be moved to Appendix?}
In \cite[\S11]{SulLong}, Sullivan described a way to compute homotopy periods.  Briefly, a \emph{homotopy period} on a space $X$ is a homomorphism $\pi_n(X) \to \RR$ which can be evaluated on a function $f:S^n \to X$ by applying the operations of pullback, primitive, and wedge to a fixed set of differential forms on $X$.  A trivial example is the degree of a map $f:S^n \to S^n$, computed as
\[\deg f=\int_{S^n} f^*d\vol_{S^n},\]
where $d\vol_{S^n}$ is a normalized volume form.  The simplest nontrivial example is Whitehead's formula for the Hopf invariant of a map $f:S^{4n-1} \to S^{2n}$, given by
\[\operatorname{Hopf}(f)=\int_{S^{4n-1}} f^*d\vol_{S^{2n}} \wedge \Prim(f^*d\vol_{S^{2n}}),\]
where $\Prim(\omega)$ designates any primitive of the form $\omega$.

More generally, Sullivan gives a construction of homotopy periods that, when $X$ is simply connected, compute any element of $\Hom(\pi_n(X),\RR)$.  Effectively, this construction produces an obstruction to extending $f:S^n \to X$ over $D^{n+1}$, or, equivalently, an obstruction to lifting it to a map $\tilde f:S^n \to \mathcal PX$, where $\mathcal PX$ is the space of based paths in $X$.  Thus one starts with a relative minimal model for the path space fibration $\mathcal PX \to X$ and constructs an extension one degree at a time.

\subsection{Background}
In order to give the precise definition of homotopy periods, we first review Sullivan's model of rational homotopy theory, as described in \cite{SulLong,GrMo,FHT}.  This model represents spaces using differential graded algebras (DGAs).  A \strong{commutative DGA} over a field $\mathbb K$ (typically $\RR$ or $\QQ$) is a cochain complex equipped with a graded-commutative multiplication satisfying the graded Leibniz rule.  Two types of DGAs are particularly important:
\begin{itemize}
    \item the algebra of smooth differential forms $\Omega^*X$ on a manifold $X$\footnote{Smooth forms on manifolds suffice for our purposes, but in the broader theory one needs a generalization.} and
    \item connected minimal DGAs.  A DGA $\mathcal A$ is \strong{minimal} if it is free as a graded commutative algebra ($\mathcal A=\Lambda V$ where $V=\bigoplus_n V_n$ is some graded vector space\footnote{More explicitly, $\Lambda V$ is the tensor product of the polynomial algebra on even-degree generators in $V$ and the exterior algebra on odd-degree generators in $V$.}) and its differential is decomposable ($dV \subseteq \Lambda^{\geq 2} V$).  It is \strong{connected} if $H_0(\mathcal A)=\mathbb K$; a minimal DGA is connected if and only if $V_0=0$.
\end{itemize}
The key observation of Sullivan is that if $X$ is simply connected (or, more generally, nilpotent) $\Omega^*X$ has a \strong{minimal model}, i.e., a homomorphism $m_X:\mathcal M_X \to \Omega^*X$ from a minimal DGA which induces isomorphisms on homology.  Moreover, $\mathcal M_X$ is unique up to (non-unique) isomorphism, and the contravariant functor $X \mapsto \mathcal M_X$ induces an equivalence of homotopy categories from \emph{rationalized} spaces to minimal DGAs.

In order to make sense of this, we need to define the notion of a homotopy of DGA homomorphisms.  Consider two DGA homomorphisms $\ph,\psi:\mathcal M \to \mathcal A$, where $\mathcal M=\Lambda V$ is minimal.  We give two equivalent definitions of a homotopy between the two:
\begin{description}
    \item[Cylinder model] Define an ``algebraic interval'', the free DGA $\Lambda(t,dt)$ with generators $t$ and $dt$ of degree $0$ and $1$, respectively.  Then a homotopy between $\ph$ and $\psi$ is a homomorphism
    \[\Phi:\mathcal M \to \mathcal A \otimes \Lambda(t,dt)\]
    such that $\Phi|_{t=0,dt=0}=\ph$ and $\Phi|_{t=1,dt=0}=\psi$.
    \item[Free path space model] For each $n$, define $W_{n-1}$ to be a vector space in degree $n-1$ equipped with an isomorphism $s:V_n \to W_{n-1}$.  We also denote $W_{n-1}$ by $sV_n$.  Observe that $s$ extends to a derivation of degree $-1$ on $\Lambda_n V_n \otimes \Lambda_{n-1} W_{n-1}$ if we set it to $0$ on $W_{n-1}$.  Then we define the free path space of $\mathcal M$ as the algebra
    \[\mathcal M^I=\mathcal M \otimes \mathcal M \otimes (\Lambda W,d_W),\]
    where
    \[d_W(sv)=v \otimes 1 \otimes 1 - 1 \otimes v \otimes 1 - 1 \otimes 1 \otimes sdv.\]
    A homotopy between $\ph$ and $\psi$ is an extension of
    \[\ph \otimes \psi:\mathcal M \otimes \mathcal M \to \mathcal A\]
    over $\mathcal M^I$.
\end{description}

\subsection{General definition}
We now give a complete account of Sullivan's construction of homotopy periods.  Let $\mathcal A=(\Lambda V_n,d)$ be a minimal DGA over $\RR$ and let $m:\mathcal A \to \Omega^*X$ be a homomorphism (which may or may not be a minimal model).  In short, the \strong{homotopy period} $\pi_v$ corresponding to an element $v \in V_n$ sends a map $f:S^n \to X$ to the obstruction to extending a partial nullhomotopy of $f^*m$ over $v$.  Following Sullivan, we now give a variation of the algebraic path space construction above which is tailored to this situation.

One defines a model for the ``path space of $\mathcal A$'' as follows.  Define $W_n$ and $s$ as before.  Then we define
\[P\mathcal A=\left(\Lambda_n V_n \otimes \Lambda_{n-1} W_{n-1},d\right)\]
where $d$ is defined inductively so that $ds+sd=\id$ (that is, we extend $d$ to $W_{n-1}$ by setting $dsv=v-sdv$ for $v \in V_n$).

Given a smooth map $f:S^n \to X$, we get a homomorphism $f^*m:\mathcal A \to \Omega^*S^n$.  Given $v \in V_n$, the homotopy period $\pi_v$ is the obstruction to extending $f^*m$ over $sv$, that is,
\[\pi_v(f)=\int_{S^n} \ph(dsv)=\int_{S^n} (f^*mv-\ph(sdv)),\]
where $\ph$ is any extension of $f^*m$ over the $W_i$.

It is purely formal to show that this is the same as the obstruction to extending a partial nullhomotopy of $f^*m$ over $v$.
\begin{prop} \label{props-of-periods}
  Given the data above, 
  %$\mathcal A=(\Lambda_n V_n,d)$ and $m:\mathcal A \to \Omega^*X$ as above, ??
  the following properties hold:
  \begin{enumerate}[(a)]
  \item The extension $\ph$ can always be defined if $\mathcal A$ is simply
    connected.
    %, meaning $H_1(\mathcal A; \mathbb{K})=0$.  
    If $\mathcal A$ is merely nilpotent, let
    $V_n=\bigcup_k V_{n,k}$ where $V_{n,k}$ is the maximal subspace such that
    $dV_{n,k}$ lies in $\Lambda_{i=1}^{n-1} V_i \otimes V_{n,k-1}$.  Then for
    $v \in V_{n,k} \setminus V_{n,k-1}$, $\ph$ is defined if the homotopy periods
    $\pi_w(f)$ are zero for $w \in V_{n,k-1}$.
  \item The homotopy period $\pi_v$ is independent of the extension $\ph$.
  \item The homotopy period $\pi_v$ is a homotopy invariant.
  \item The homotopy period $\pi_v$ is a homomorphism from its domain of definition (a subgroup of $\pi_n(X)$) to $\RR$.
  \end{enumerate}
\end{prop}
\begin{proof}
  For part (a), notice that for $i<n$, there is no obstruction to
  extending $f^*m$ to $W_{i-1}$.  Thus the only relevant obstruction is that to
  extending $f^*m$ over $W_{n-1}$.

  Part (b) is proved using homotopy theory of DGAs.  Suppose $\ph$ and
  $\psi$ are two different extensions.  We attempt to extend the constant
  homotopy
  \[f^*m \otimes 1:\mathcal A \to \Omega^*S^n \otimes \Lambda(t,dt)\]
  to a homotopy from $\ph$ to $\psi$.  By \cite[Prop.~10.4]{GrMo}, for $w \in W_{i-1}$, the obstruction to
  extending the homotopy to $w$ lies in
  \[H^i(\Omega^*S^n \otimes \Lambda(t,dt), \Omega^*S^n \oplus \Omega^*S^n;\RR) \cong H^{i-1}(\Omega^*S^n; \RR),\]
  which is zero since $i \leq n$.  Therefore there is a homotopy $\Phi$ from
  $\ph$ to $\psi$.  Then
  \[\int_{S^n} \psi(dsv)=\int_{S^n} \left(\ph(sdv)-d{\textstyle\int_0^1} \Phi(dsv)+{\textstyle\int_0^1} d\Phi(dsv)\right)=\int_{S^n} \ph(dsv).\]

  For part (c), to show that $\pi_v$ is a homotopy invariant, consider a homotopy
  $f_t:[0,1] \times S^n \to X$, and let $\ph_0$ be an extension of $f_0^*m$ over
  the relevant $W_i$.  Then an extension of $f_1^*m$ is given inductively by
  \[\ph_1(sv)=\ph_0(sv)+\int_0^1 f_t^*mv,\]
  where $\int_0^1$ denotes the fiberwise integral.  From the identity
  \[d\int_0^1 f_t^*\omega+\int_0^1 df_t^*\omega=f_1^*\omega-f_0^*\omega,\]
  it follows that this is indeed a DGA homomorphism.  Then
  \begin{align*}
    \pi_v(f_1) &= \int_{S^n} (f_1^*mv-\ph_1(sdv)) \\
               &= \int_{S^n} \left(f_1^*mv-\ph_0(sdv)-{\textstyle\int_0^1} f_t^*mdv\right) \\
               &= \int_{S^n} \left(f_0^*mv-\ph_0(sdv)-d{\textstyle\int_0^1} f_t^*mv\right) \\
    &= \pi_v(f_0).
  \end{align*}

  Finally, for part (d), to show that $\pi_v$ is a homomorphism, notice that if
  $\mu:S^n \to S^n \vee S^n$ is the co-H-space operation on $S^n$, then given
  extensions $\ph_1$ of $f_1^*m$ and $\ph_2$ of $f_2^*m$ over the $W_i$,
  $\mu^*(\ph_1 \vee \ph_2)$ is an extension of $\mu^*(f_1 \vee f_2)^*m$.
\end{proof}

\begin{ex}
  We now demonstrate that the homotopy period associated with a map
  $f:S^{4n-1} \to S^{2n}$ (which is unique up to a multiplicative constant) 
  is Whitehead's formula for the Hopf invariant given above.  The minimal model
  of $S^{2n}$ is
  \[\mathcal A=\left(\Lambda(a^{2n},b^{4n-1}),da=0,db=a^2/2\right),\]
  and a homomorphism $m:\mathcal A \to \Omega^*S^{2n}$ is given by
  $m(a)=d\vol_{S^{2n}}$ and $m(b)=0$.  Then we can extend $f^*m$ to $sa$ by
  assigning it any primitive of $f^*d\vol_{S^{2n}}$, and the homotopy period
  $\pi_b$ is given by
  \begin{align*}
    \pi_b(f) &= \int_{S^{4n-1}}(f^*mb-\ph(s(a^2/2))) \\
           &= \int_{S^{4n-1}} \ph(a \cdot s(a)) \\
           &= -\int_{S^{4n-1}} f^*d\vol_{S^{2n}} \wedge \Prim(f^*d\vol_{S^{2n}})
             = -\operatorname{Hopf}(f).
  \end{align*}
\end{ex}

\subsection{Homotopy periods for wedges of spheres} \label{S:wedges}
Now we specialize to the case that the space $X$ is homotopy equivalent to a wedge of spheres:
\[X \simeq S^{q_1} \vee \cdots \vee S^{q_s}.\]
Suppose first that $X$ is simply connected, i.e., $q_i>1$ for each $i$.  Then by the Hilton--Milnor theorem \cite{Hilton}, $\pi_*(X) \otimes \QQ$ is the free graded Lie algebra on generators of degrees $q_1,\ldots,q_s$; the graded dual of this space is the vector space of indecomposable generators of the minimal model.  Moreover, the pairing can be computed using a formula proved by Andrews and Arkowitz \cite[Theorem 6.1]{AA} relating Whitehead products and the differential in the minimal model: for $\alpha \in \pi_k(X),\beta \in \pi_\ell(X)$ and $x$ a $(k+\ell-1)$-dimensional indecomposable in the minimal model of $X$,
\begin{equation} \label{eq:AA}
  x([\alpha,\beta])=\sum_i C_i((-1)^{k \ell}y_i(\alpha)z_i(\beta)+y_i(\beta)z_i(\alpha)),
\end{equation}
%\noteblue{I changed $(-1)^{pq}$ to $(-1)^{k\ell}$---correct?}
where $i$ runs over quadratic terms $C_iy_iz_i$ of $dx$.  The differential (which, for wedges of spheres, turns out to be purely quadratic) can therefore be thought of as a comultiplication on $V^*:=\Hom(\pi_*(X),\RR)$ dual to the Whitehead product, giving the vector space of indecomposables the structure of a cofree Lie coalgebra.

These dual objects were studied in more detail by Sinha and Walter \cite{SW1,SW2}, who gave combinatorial descriptions of bases and the pairing between them.  Write $\overline x=(-1)^{\deg x}x$.  Then $V^*$ is spanned by elements of the form $x_I$, where $I=(i_1,\ldots,i_r)$ is a multiindex, $x_{(i)}$ is dual to the inclusion of $S^{p_i}$, and
\[dx_I=\sum_{k=1}^{r-1} \overline{x_{(i_1,\ldots,i_k)}}x_{(i_{k+1},\ldots,i_r)}.\]
We say $r$ is the \strong{depth} of the homotopy period $x_I$.\footnote{Outside of \S\ref{S:periods}, we use $d$ for the depth of the relevant homotopy periods and $r$ for the number of link components.  In \S\ref{S:Milnor-and-Massey}, these two quantities are equal.}  
Moreover \cite[Proposition~3.21]{SW1}, the $x_I$ satisfy the \emph{shuffle relations}
\[\sum_\sigma (-1)^{\kappa(\sigma)}x_{\sigma(I,J)}=0,\]
where $\sigma$ runs over all \emph{shuffles} of $I$ and $J$, i.e., sequences in which $I$ and $J$ are subsequences and which are the union of these two subsequences, and the \emph{Koszul sign} $\kappa(\sigma)$ is given by
\[\kappa(\sigma)=\sum_{i \in I, j \in J: \sigma(j)<\sigma(i)} (q_i-1)(q_j-1).\]

Applying the Andrews--Arkowitz formula recursively, one sees for example that $x_{(1,\ldots,s)}$ pairs nontrivially with $[\iota_1,[\iota_2,\cdots[\iota_{s-1},\iota_s] \cdots ]]$, where $\iota_i \in \pi_{q_i}(X)$ is the inclusion of $S^{q_i}$.  More generally, if $S \subset \Sigma_s$ is the subgroup of permutations of $\{1,\ldots,s\}$ fixing $s$, then $\{x_\sigma:=x_{\sigma(1),\ldots,\sigma(s-1),s}: \sigma \in S\}$ and
\[\{\iota_\sigma:=[\iota_{\sigma(1)},[\iota_{\sigma(2)},\cdots[\iota_{\sigma(s-1)},\iota_s] \cdots ]] : \sigma \in S\}\]
form dual bases for subspaces of $V^*$ and $\pi_*(X) \otimes \RR$.  See for example work of Walter \cite[Example~2.19]{Walter} for further calculations.

The Sullivan minimal model can be constructed for any space $X$, although it is no longer necessarily of finite type, nor is it clear exactly what homotopical information it preserves (beyond nilpotent quotients of the fundamental group) when the fundamental group of $X$ is not nilpotent.  In particular, we can consider the Sullivan model of a non--simply connected wedge of spheres (or another space which is related to it by a zigzag of rational homology equivalences).  It has essentially the same structure, and in particular its indecomposables still form a cofree Lie coalgebra \cite[Theorem 1]{FH}, although the free Lie algebra (generated by the $s$ spheres) to which it is dual is somewhat harder to interpret geometrically.

\subsection{An alternate computation}
Now, again following Sullivan, we modify the construction of homotopy periods to get one that is equivalent, but somewhat more general, using an alternate model for the path space.  We use the same additional generators $W_n=sV_n$, but instead of defining $d$ so that $ds+sd=\id$, we ensure $ds+sd=\Delta$, where $\Delta$ is some isomorphism in positive degrees.  Proposition \ref{props-of-periods} still holds after this modification.  In the case of a wedge of spheres 
$X \simeq S^{q_1} \vee \cdots \vee S^{q_s}$, for a fixed $I=(i_1,\ldots,i_r)$, we can define $d$ so that for every $1 \leq k \leq r$,
\begin{align}
  dsx_{i_k,\ldots,i_r} &= x_{i_k,\ldots,i_r}-\sum_{\ell=k}^{r-1} x_{i_k,\ldots,i_\ell}(sx_{i_{\ell+1},\ldots,i_r}) \\
  \Delta x_{i_k,\ldots,x_r} &= x_{i_k,\ldots,i_r}+\sum_{\ell=k}^{r-1} (s\overline{x_{i_k,\ldots,i_\ell}})x_{i_{\ell+1},\ldots,i_r}.
\end{align}
It suffices to check that this satisfies $d^2=0$; one can then extend $\Delta$ arbitrarily to the rest of the indecomposables, for example by setting it to be the identity on a complementary subset.  Now given a homomorphism $m:\mathcal A \to \Omega^*X$ and a map $f:S^n \to X$, where $n=q_{i_1}+\cdots+q_{i_r}-r+1$, 
%\noteblue{(Should all the $p_i$'s be $q_i$'s?  Is $X$ still $S^{q_1} \vee \dots \vee S^{q_r}$?)}
an alternate computation of the homotopy period $\pi_{x_{i_k,\ldots,i_r}}(f)$ is given by
\begin{equation} \label{eq:Delta}
  \pi_{x_{i_1,\ldots,i_r}}(f)=\int_{S^n} \ph(dsx_{i_1,\ldots,i_r})=\int_{S^n} \Bigl(f^*mx_{i_1,\ldots,i_r}-\sum_{k=1}^{r-1} f^*mx_{i_1,\ldots,i_k}\ph(sx_{i_{k+1},\ldots,i_r})\Bigr),
\end{equation}
with $\ph$ again constructed inductively so that
\[\ph(sx_{i_k,\ldots,i_r})=\Prim(\ph(\Delta v-sdv))=\Prim\Bigl(f^*mx_{i_k,\ldots,i_r}-\sum_{\ell=k}^{r-1} f^*mx_{i_k,\ldots,i_\ell}\ph(sx_{i_{\ell+1},\ldots,i_r})\Bigr).\]

\subsection{Wedges of circles and generalizations}
Now suppose that $X$ is the complement of a collection of $k$ embedded copies of $S^{m-2}$ inside $S^m$, for any $m \geq 4$.   Consider a map $i:\bigvee_k S^1 \to X$ which takes each circle to a small loop linked with the corresponding $S^{m-2}$.
Let $F_k$ denote the free group on $k$ generators, and for any group $\Gamma$, let $\Gamma_r = [\Gamma, \Gamma_{r-1}]$ denote the $r$th stage in the lower central series of $\Gamma$, starting with $\Gamma_1=\Gamma$.
\begin{prop} \label{prop:almost-free}
    The map $i$ induces an isomorphism $F_k/(F_k)_r \to \pi_1(X)/(\pi_1(X))_r$ on the quotients of the fundamental groups by any stage of their lower central series.
\end{prop}
\begin{proof}
    By Alexander duality, $X$ has $H_1$ generated by the Alexander duals of the $S_j$ and $H_2 \cong 0$, so $i$ induces isomorphisms on both $H_1$ and $H_2$.  By a theorem of Stallings \cite[Theorem 5.1]{Sta}, it induces isomorphisms on all lower central series quotients.
\end{proof}
We consider homotopy periods of loops in $X$; although $X$ is not a nilpotent space and does not have a minimal model, we can still use homotopy periods to compute invariants of the fundamental group.

When $X$ is such that $\Gamma=\pi_1(X)$ is a free group, the homotopy periods on $X$ as redefined in \eqref{eq:Delta} are exactly the \emph{letter-linking invariants} described by Monroe and Sinha \cite{MonSin}, who show that they are equal to the corresponding coefficients of the Magnus expansion \cite{Magnus2}; see also \cite{Gadish,GOSW} for generalizations.  In particular, \cite[Theorem 5.1]{MonSin} states that homotopy periods of depth $r$ are dual to $(\Gamma_r/\Gamma_{r+1}) \otimes \QQ$.
%, where $\Gamma_r=[\Gamma,\Gamma_{r-1}]$ is the $r$th step of the lower central series, starting with $\Gamma_1=\Gamma$.

Now suppose $\Gamma$ is not free, but only admits a map from $F_k$ 
%the free group $F_k$ on $k$ generators 
which induces isomorphisms on lower central series quotients $F_k/(F_k)_r \to \Gamma/\Gamma_r$ for all $r$.  In this case, homotopy periods enjoy the same properties:
\begin{prop} \label{prop:period=Magnus}
    Homotopy periods of depth $r$ are dual to $(\Gamma_r/\Gamma_{r+1}) \otimes \QQ$, where $\Gamma=\pi_1(X)$.  Moreover, their values on elements are given by coefficients of the Magnus expansion.
\end{prop}
\begin{proof}
    Define a pair of homomorphisms
    \[F_k \xrightarrow{i_*} \Gamma \to \Gamma/\Gamma_{r+1} \cong F_k/(F_k)_{r+1}\]
    whose composition is the quotient map, induced by maps
    \[\bigvee_k S^1 \xrightarrow{i} X \xrightarrow{f} N,\]
    where $N$ is a nilmanifold.  Since $f \circ i$ induces an isomorphism on the first $r$ stages of the minimal model, any homotopy period in $\bigvee_k S^1$ of depth up to $r$ can be computed by doing the computation in $N$ and then pulling back.  Thus the value of a homotopy period of depth $r$ on an element of $\Gamma$ depends only on its image in $\Gamma/\Gamma_{r+1}$.

    The statement about the Magnus expansion follows from the discussion immediately before the Proposition statement.
\end{proof}

\subsection{Pairing with Whitehead products in the general case}
For the sake of \S\ref{S:exp} and \S\ref{S:Massey-Milnor-Koschorke}, we would like to apply an analogue of the Andrews--Arkowitz formula \eqref{eq:AA} in a more general context, where some of the $q_i$ are $1$ and others are not.  To eschew overly technical hypotheses, we will give a result narrowly tailored to our situation.  We suspect that one can make a more general statement by relating the homotopy groups of the wedge to the homotopy Lie algebra discussed in \cite{FH}, but we do not pursue it here.
\begin{prop} \label{Whitehead/period}
    Let $I=(1,\ldots,r)$ and let $\alpha_i$ be the inclusion of $S^{q_i}$ in $S^{q_1} \vee \cdots \vee S^{q_r}$.  Then for a permutation $\sigma$ of $\{1,\ldots,r-1\}$,
    \[x_I([\alpha_{\sigma(1)},[\alpha_{\sigma(2)},\cdots[\alpha_{\sigma(r-1)},\alpha_r] \cdots]])=\pm 1\quad\text{if }\sigma=\id,\qquad 0\quad\text{otherwise}.\]
\end{prop}
When each of the $q_i$ is at least $2$, this follows immediately from formula \eqref{eq:AA}.  In the general case, it follows by induction from the following lemma:
\begin{lem} Let $I=(i_1,\ldots,i_r)$.  
%Then the following statements hold:
    \begin{enumerate}[(a)]
    \item \label{pi1-AA} If $r \geq 2$ and $j \neq i_k$ for $k \geq 2$, we have
        \[x_I([\alpha_j,\beta])=\pm x_{i_1}(\alpha_j)x_{(i_2,\ldots,i_r)}(\beta),\]
        so long as $x_{(i_2,\ldots,i_r)}(\beta)$ is defined.
    \item \label{diff-heights} If $s \geq 2$, then
        \[x_I([\alpha_{j_s},\cdots[\alpha_{j_{r-1}},\alpha_{j_r}] \cdots ])=0.\]
    \end{enumerate}
\end{lem}
\begin{proof}
    We use induction on $r$ and the formula \eqref{eq:Delta}.

    First, part \ref{diff-heights} is true for $r=1$ because $x_{(i)}$ is a cocycle and Whitehead products are always homologically trivial.  For $r \geq 2$, part \ref{diff-heights} follows from part \ref{diff-heights} for $r-1$ and part \ref{pi1-AA} for $r$.

    It remains to prove part \ref{pi1-AA}.  We assume by induction that both parts hold for $|I|<r$.  Write $n=\deg \beta$.

    Denote the minimal model of the wedge by $m:\mathcal A \to \Omega^*(S^{q_1} \vee \cdots \vee S^{q_r})$; we can assume that $m$ is nonzero only on the $x_{(i)}$.  Write $f:S^n \to S^{q_1} \vee \cdots \vee S^{q_r}$ for the standard representative of $[\alpha_j,\beta]$; in particular, $f$ factors through
    \[S^n \xrightarrow{\partial} S^{q_j} \vee S^{n-q_j+1} \xrightarrow{\alpha_j \vee \beta} S^{q_1} \vee \cdots \vee S^{q_r}.\]
    We would like to compute the obstruction to nullhomotoping $f^*m$.  Since this is a homotopy invariant of homomorphisms $\mathcal A \to \Omega^*S^n$, it is enough to compute it for $\partial^*(\alpha_j^*m \oplus b)$, where $b$ is a homomorphism homotopic to $\beta^*m$.

    Since $x_{i_2},\ldots,x_{i_r}$ evaluate to zero on $S^{q_j}$, $\alpha_j^*m$ sends them to zero, as well as $x_J$ for every tuple $J$ of elements not equal to $j$.  It sends $x_j$ to the volume form $d\vol_{S^{q_j}}$.

    By induction, we can construct a homotopy of $\beta^*m$ to a homomorphism which sends $x_J \mapsto 0$ for every $s$-tuple $J$ of distinct elements in $i_2,\ldots,i_r$, for $s<r-1$.  This map extends to one which sends $x_{(i_2,\ldots,i_r)} \mapsto x_{(i_2,\ldots,i_r)}(\beta)d\vol_{S^{n-q_j+1}}$.

    By formula \eqref{eq:Delta}, the obstruction to nullhomotoping $\theta=\partial^*(\alpha_j^*m \oplus b)$ is given by
    \[-\int_{S^n} \sum_{k=1}^{r-1}\theta(x_{i_1,\ldots,i_k})\ph(sx_{i_{k+1},\ldots,i_r}),\]
    with $\ph$ constructed inductively.  Following the induction, we see that for $k>1$, $\ph(sx_{i_{k+1},\ldots,i_r})=0$, and for $k=1$ it is a primitive of $\partial^*b(x_{(i_2,\ldots,i_r)})$.  Therefore, we get
    \begin{align*}
        x_I([\alpha_j,\beta]) &= -\int_{S^n} \partial^*(\alpha_j^*m \oplus 0)(x_{i_1}) \wedge \Prim(\partial^*(0 \oplus b)(x_{(i_2,\ldots,i_r)})) \\
        &= -x_{i_1}(\alpha_j)x_{(i_2,\ldots,i_r)}(\beta)\int_{S^n} \partial^* d\vol_{S^{q_j}}\Prim(\partial^*d\vol_{S^{n-q_j+1}}).
    \end{align*}
    The last integral is, up to sign, the linking number of preimages of arbitary points of the two spheres, i.e., $\pm 1$.
\end{proof}

\section{Milnor invariants and Massey products}\label{S:Milnor-and-Massey}

In the classical setting of $1$-dimensional links in $S^3$, Turaev \cite{Turaev}
and later Porter \cite{Porter} showed that Milnor's $\mubar$-invariants can be
computed via Massey products.  This Massey product definition of Milnor invariants extends verbatim to higher dimensions, as we discuss next.  In this section, we restrict to Milnor invariants with distinct indices, so the number of components $r$ is the depth $d$ of the Milnor invariant.

\subsection{Massey products} \label{S:Massey}
Massey products are a family of higher cohomology operations generalizing the cup product.  One starts with the Massey triple product of three cohomology classes, which can be defined when their pairwise cup products are zero.  In general, Massey products are non-unique because the definition includes a choice of primitive: the $d$-fold Massey product of a $d$-tuple is defined as a \emph{set} of cohomology classes, which is nonempty when certain lower-order Massey product sets contain zero.

We now give the formal definition.  Given a space $X$, let $u_i \in H^{q_i}(X;R)$, for $i=1,\ldots,d$ be cohomology classes with coefficients in a ring $R$.  Then the \strong{Massey product} $\langle u_1,\ldots,u_d \rangle$ is defined as follows.  Let $a_{i..i}=u_i$, and for $1 \leq i<j \leq d$, suppose that there exist cochains
$a_{i..j}$ such that
\[
\delta a_{i..j}=\sum_{k=i}^{j-1} \overline{a_{i..k}}a_{k+1..j}
\]
where $\overline x=(-1)^{\deg x}x$.  Then $\langle u_i,\ldots,u_d \rangle$ is the set of all possible cohomology classes of
\[\sum_{k=i}^{d-1} \overline{a_{i..k}}a_{k+1..d}.\]
In particular, the two-fold Massey product is just the cup product.  In general, the $d$-fold Massey product lies in degree
\begin{equation}
\label{eq:deg-masssey}
\sum_{i=1}^d q_i-(d-2).
\end{equation}

More generally, we can take $u_i \in H^{q_i}(X_i;R)$, where the $X_i$ are subspaces of a common space $X$.  Then the Massey product $\langle u_1,\ldots,u_d \rangle$ is a class in $H^*(X_1 \cap \cdots \cap X_d;R)$.  This is the construction used by Porter to give an alternate definition of Milnor invariants, which we repeat here.  %(Note that Porter uses a different sign convention for Massey products, and therefore our sign convention for Milnor invariants also differs from his.  In the case of the classical Milnor invariants, our conventions differ by a factor of $(-1)^d$.)  \noteblue{Robin: The comment about signs also appears in a Remark below.  Do we want it in both places?}

\subsection{Definition of Milnor invariants}
Consider a smooth embedding
\begin{equation}\label{eq:emb-link-map-f}
  f:S^{p_i} \sqcup \cdots \sqcup S^{p_d} \longrightarrow S^m=\RR^m\cup \{\infty\},
\end{equation}
where $1 \leq p_1,\ldots,p_d \leq m-2$.  As before, define $q_i=m-p_i-1$.  Let $N_i$ be disjoint tubular neighborhoods of $f(S^{p_i})$, and let $u_i \in H^{q_i}(S^m \setminus N_i;R)$ be the Alexander dual of $[f(S^{p_i})]$.  We consider the Massey product $\langle u_1,\ldots,u_d \rangle$. %First, note that the degree of a Massey product is always greater than that of any of the factors.  Therefore, 
Using its degree from equation \eqref{eq:deg-masssey} and the cohomology of $S^m \setminus \img f$ via Alexander duality, we deduce the following fact:
\begin{prop}
    The Massey product $\langle u_1,\ldots,u_d \rangle$ is zero unless it is $(m-1)$-dimensional, i.e., when
    \[
    \pushQED{\qed}
    \sum_{i=1}^d p_i=(m-2)(d-1)+1.
    \qedhere
    \popQED
    \]
\end{prop}
For $m=3$ and $p_1=\cdots=p_d=1$, Porter defines the Milnor invariant $\mubar(1,\ldots,d) \in R$ to be such that the Massey product $\langle u_1,\ldots,u_d \rangle$ consists of the single element
\[\mubar(1,\ldots,d)v_{1,d},\]
where $v_{i,j}$ is the Lefschetz dual of a path between the $i$th and $j$th spheres.  In the classical case, he shows that this is well-defined whenever the Milnor invariants corresponding to subsequences of $(1,\ldots,d)$ are zero.  Below, we extend this result to all possible $m$ and $p_1,\ldots,p_d$.

Assuming for now that this gives a well-defined invariant, we get a relatively simple integral formula for it:
\begin{prop} \label{simpler}
    Define the following differential forms:
    \begin{itemize}
    \item $\xi_i \in \Omega^{q_i+1}(S^m)$, a Poincar\'e dual to $f(S^{p_i})$ supported on $N_i$,
    \item $\omega_i=\omega_{i..i}\in \Omega^{q_i}(S^m)$, a primitive in $S^m$ of $\xi_i$,
    \item $\omega_{i..j}$, a primitive in $S^m \setminus (N_i \cup \cdots \cup N_j)$ of
    \[d\omega_{i..j}=\sum_{k=i}^{j-1} \overline{\omega_{i..k}}\omega_{k+1..j}, \text{ and} \]
    \item $\theta_{1,d} \in \Omega^1(S^m \setminus (N_1 \cup \cdots \cup N_d))$, a representative of the Lefschetz dual of $[\partial N_1]$.
    \end{itemize}
    Then
    \[
    \pushQED{\qed} 
    \mubar(1,\ldots,d)=\int_{S^m \setminus (N_1 \cup \cdots \cup N_d)} \theta_{1,d}\omega_{1..d}.
    \qedhere
    \popQED
    \]
\end{prop}
This formula follows from the definition of $\mubar(1,\ldots,d)$ and will be useful to us in some cases.  In particular, it extends verbatim to the case in which some indices are repeated.  However, for many results, including the proof that the Massey product takes the stated form, we need a more complicated formula:

\begin{thm} \label{thm:int-formula}
  Let $f=(f_1,\dots, f_d):S^{p_i} \sqcup \cdots \sqcup S^{p_d} \to S^m$ be a smooth embedding,
  where $1 \leq p_1,\ldots,p_d \leq m-2$.  Let $N_i$ be disjoint tubular
  neighborhoods of $f(S^{p_i})$, and let $u_i \in H^{q_i}(S^m \setminus N_i)$
  be the Alexander dual of $[f(S^{p_i})]$.  Suppose that all Massey products
  \[\langle u_{i_1},\ldots,u_{i_r} \rangle \subseteq H^*(S^m \setminus (N_{i_1} \cup \cdots \cup N_{i_r}))\]
  are zero for every proper subsequence $i_1,\ldots,i_r$ of $1,\ldots,d$, and that
  \[\sum_{i=1}^d p_i=(m-2)(d-1)+1.\]
  Then the Massey product $\langle u_1,\ldots,u_d \rangle$ consists of a single element
  \[\mubar(1,\ldots,d)v_{1,d},\]
  where $v_{i,j}$ is the Lefschetz dual of a path between the $i$th and $j$th
  spheres.  Finally, the number $\mubar(1,\ldots,d)$ can be computed by
  \begin{equation} \label{eqn:Massey}
    \mubar(1,\ldots,d)=\int_{N_d} \sum_{k=1}^{d-1}\overline{\omega_{1..k}}\ph_{k+1..d}^l\xi_d,
  \end{equation}
  where one defines the following differential forms:
  \begin{itemize}
  \item Let $\xi_i \in \Omega^{q_i+1}(S^m)$ be a Poincar\'e dual to $S^{p_i}$
    supported on $N_i$.
  \item Let $\omega_i=\omega_{i..i} \in \Omega^{q_i}(S^m)$ be a primitive of $\xi_i$, and let $\ph_{i..i}^l=\ph_{i..i}^r=1 \in \Omega^0(S^m)$.
  \item For $i<j$, $(i,j) \neq (1,d)$, define $\omega_{i..j}$ as a primitive in
    $S^m$ of
    \begin{equation}
        \label{eq:d-omega-ij}
        d\omega_{i..j}=\sum_{k=i}^{j-1} \overline{\omega_{i..k}}\omega_{k+1..j} + \sum_{k=i}^j \overline{\ph_{i..k}^l}\xi_k\ph_{k..j}^r
    \end{equation}
    where $\ph_{i..j}^l$ and $\ph_{i..j}^r$ are primitives in $\Omega^*(N_j)$ and
    $\Omega^*(N_i)$, respectively, of
    \begin{equation} \label{eq:d-ph-ij}
        d\ph_{i..j}^l=\sum_{k=i}^{j-1} \overline{\omega_{i..k}}\ph_{k+1..j}^l \quad\text{and}\quad d\ph_{i..j}^r=\sum_{k=i}^{j-1} \overline{\ph_{i..k}^r}\omega_{k+1..j}.
    \end{equation}
  \end{itemize}
\end{thm}
Submanifolds dual to the forms defined above are shown in an example in Figure \ref{fig:borr-rings}.
\begin{defn}
\label{D:Milnor-invt}
  The number $\mubar(1,\ldots,d)$, defined as in Theorem \ref{thm:int-formula} as the coefficient of $v_{1,d}$ in the Massey product $\langle u_1,\ldots,u_d \rangle$, is the \strong{Milnor invariant}
  of the ordered $d$-tuple of embedded spheres determined by the embedding  
  $f:S^{p_i} \sqcup \cdots \sqcup S^{p_d} \to S^m$.
  %corresponding to the ordered tuple of embedded spheres.
\end{defn}
\begin{rmk}
  Porter \cite[Theorem 3]{Porter} shows that this definition recovers
  that of Milnor invariants of classical links in $S^3$.  Note that our formula
  differs from Porter's by a factor of $(-1)^d$ since we use a different sign
  convention for Massey products.
\end{rmk}
\begin{rmk}
  Yet another integral formula interpolates between those of Proposition \ref{simpler} and Theorem \ref{thm:int-formula}: as in Proposition \ref{simpler}, define the forms $\omega_{i..j}$, $j<d$, on the complement of the tubular neighborhoods $N_i,\ldots,N_j$, but then define the Milnor invariant via equation \eqref{eqn:Massey}.  As we discuss below, this gives an alternate definition of the Milnor invariant as a homotopy period of $f_d$ in $S^m \setminus (N_1 \cup \cdots \cup N_{d-1})$.  In the classical case, this alternate definition was already mentioned by Porter in the guise of ``functional Massey products'', and in codimension $\geq 3$ it coincides with those of Haefliger and Steer \cite{HaeSt} and Turaev \cite{Turaev2}.

  Our more complicated formula has multiple advantages.  On the topological
  front, it allows us to show that the Massey product evaluates to zero on
  $\partial N_i$ for $i \neq 1,d$.  Geometrically, it avoids taking
  primitives in the (geometrically rather mysterious) complement, instead
  taking them in $S^m$ (whose geometry is fixed) and the $N_i$ (whose geometry
  is controlled).
\end{rmk}
\begin{rmk}
  While we express everything in terms of differential forms and real coefficients, Theorem \ref{thm:int-formula} can be proved the same way using simplicial or singular cochains in any coefficient ring.
\end{rmk}
\begin{proof}[Proof of Theorem \ref{thm:int-formula}]
  A straightforward calculation shows that the forms defined in \eqref{eq:d-omega-ij} and \eqref{eq:d-ph-ij} are closed, under the inductive assumption that the forms used within them are well-defined and their differentials also satisfy \eqref{eq:d-omega-ij} and \eqref{eq:d-ph-ij}.  We use this in several steps below.

  Assume that the primitives defining $\omega_{i..j}$ and $\ph_{i..j}^{l/r}$ exist for $1 \leq i \leq j \leq d$ when $(i,j) \neq (1,d)$.  Restricting to the subspace $S^m \setminus (N_i \cup \cdots \cup N_j)$, we get
  \[d\omega_{i..j}=\sum_{k=i}^{j-1} \overline{\omega_{i..k}}\omega_{k+1..j},\]
  since the $\xi_k, i \leq k \leq j$, are zero on this subspace; hence $\sum_{k=1}^d \overline{\omega_{1..k}}\omega_{k+1..d}$ restricts on $S^m \setminus (N_1 \cup \cdots \cup N_d)$ to a representative of an element of $\langle u_1,\ldots,u_d \rangle$.
  
  By Stokes' theorem and the fact that $\xi_j$ is zero on $N_i$ for $i\neq j$, we get that for $1<i<d$,
  \begin{align*}
    \langle u_1,\ldots,u_d \rangle(\partial N_i) &= \int_{\partial N_i} \sum_{k=1}^{d-1} \overline{\omega_{1..k}}\omega_{k+1..d} \\
    &= \int_{N_i} d\left(\sum_{k=1}^{d-1} \overline{\omega_{1..k}}\omega_{k+1..d}\right) \\
    &= -\int_{N_i} d\left(\overline{\ph_{1..i}^l}\xi_i\ph_{i..d}^r\right)  \quad \text{(since \eqref{eq:d-omega-ij} is closed)}\\
    &= -\int_{\partial N_i} \overline{\ph_{1..i}^l}\xi_i\ph_{i..d}^r = 0 \quad \text{(since $\xi_i$ is supported on $N_i$)}.
  \end{align*}
  By Alexander duality, $H^{m-1}(S^m \setminus (N_1 \cup \cdots \cup N_d)) \cong \widetilde H_0(N_1 \cup \cdots \cup N_d)$, and therefore $\langle u_1,\ldots,u_d \rangle$ must be a multiple of $v_{1,d}$.  Again by Stokes' theorem, the coefficient is
  \[\int_{\partial N_i} \sum_{k=1}^{d-1} \overline{\omega_{1..k}}\omega_{k+1..d}=\int_{N_d} \sum_{k=1}^{d-1}\overline{\omega_{1..k}}\ph_{k+1..d}^l\xi_d.\]
  (These are the only surviving terms since terms of the form $\pm \omega_{1..k}\omega_{k+1..\ell}\omega_{\ell+1..d}$ cancel out, and terms including a $\xi_i$ for $i<d$ are zero on $N_d$.)
  
  Finally we need to show that the primitives $\omega_{i..j}$ and $\ph_{i..j}^{l/r}$ exist for all $(i,j) \neq (1,d)$.  We have already remarked that the forms \eqref{eq:d-omega-ij} and \eqref{eq:d-ph-ij} are closed.  Then $d\omega_{i..j}$ is exact since its degree is always $\leq m-1$.  On the other hand, $d\ph_{i..j}^l$ may be a $p_j$-form, in which case in order to show that it is exact we must show that its integral over $f_j(S^{p_j})$ is zero.  Similarly, $d\ph_{i..j}^r$ may be a $p_i$-form.  However, by again using Stokes' theorem, equation \eqref{eq:d-omega-ij}, and the fact that $\xi_j$ is zero on $N_i$ when $i\neq j$, we get 
  \begin{align*}
    \int_{f_j(S^{p_j})} d\ph_{i..j}^l &= \int_{N_j} d\ph_{i..j}^l\xi_j \\
    &= \int_{\partial N_j} \sum_{k=i}^{j-1} \overline{\omega_{i..k}}\omega_{k+1..j} \\
    &= \langle u_i,\ldots,u_j \rangle(\partial N_j).
  \end{align*}
  Similarly
  \[\int_{f(S^{p_i})} d\ph_{i..j}^r=\langle u_i,\ldots,u_j \rangle(\partial N_i).\]
  %By induction (on $d$?), 
  Thus either integral computes the Massey product $\langle u_i,\ldots,u_j \rangle$, which we have assumed to be zero.
\end{proof}

%\warning{To do: give an illustration of this calculation for the Borromean rings using the Poincar\'e/Lefschetz duals of the various forms.}

\begin{figure} %[htbp]
    \begin{tikzpicture}
    \node at (0,0)
    {\includegraphics[scale=0.3]{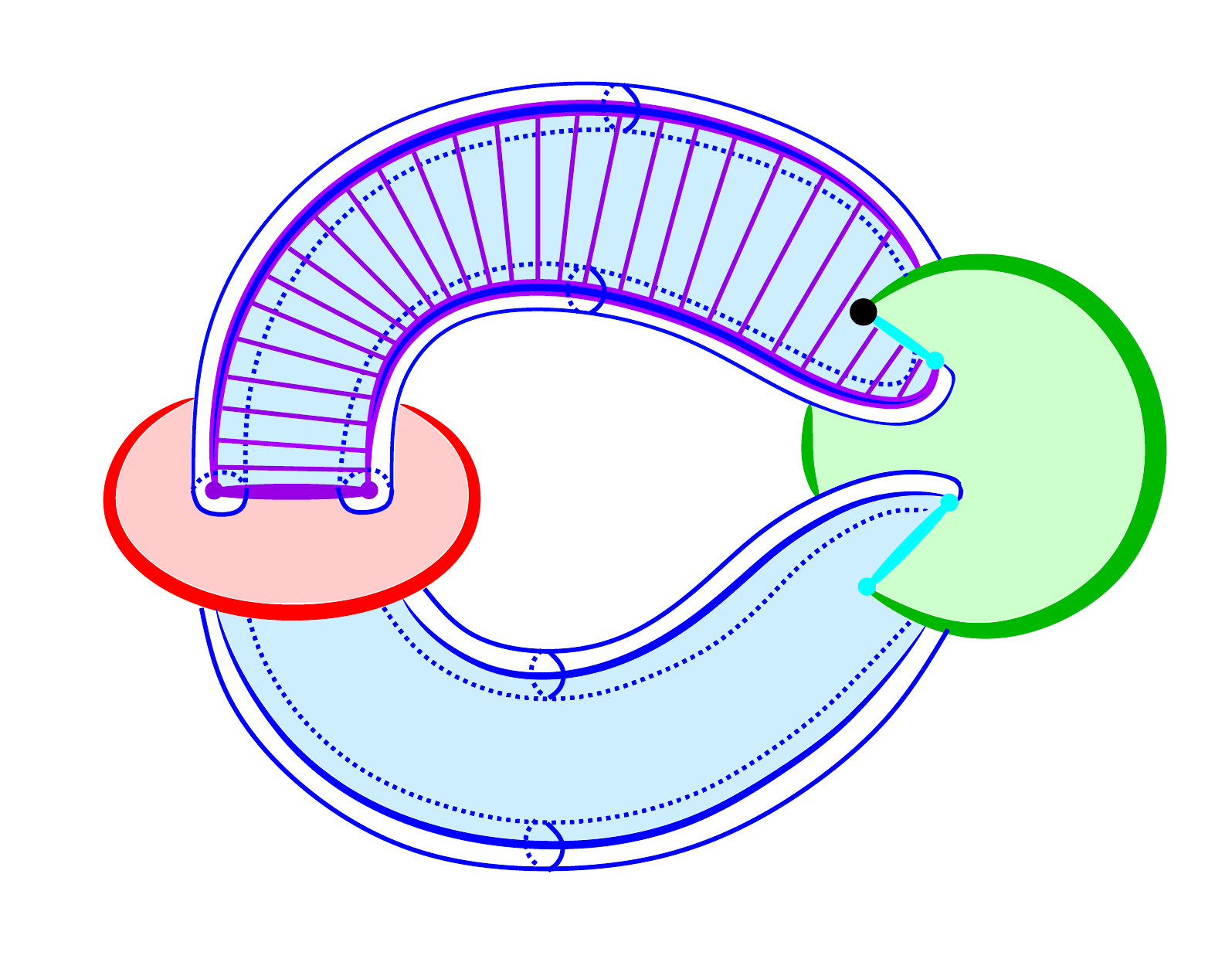}};
    \node at (-4.25,1.25) {$\omega_{1..1}$};
    \draw[->, thick] (-4,1) -- (-3,0);
    \node at (-5,-0.25) {$\xi_1$};
    \draw[->, thick] (-4.75,-0.25) -- (-3.35,-0.25);
    \node at (4.5,1) {$\omega_{3..3}$};
    \draw[->, thick] (4,0.9) -- (2.75,0.25);
    \node at (5,0) {$\xi_3$};
    \draw[->, thick] (4.75,0) -- (3.6,0);
    \node at (-2.35,-2.95) {$\omega_{2..2}$};
    \draw[->, thick] (-2.25,-2.75) -- (-1.5,-1.5);
    \node at (-1.5,-3.25) {$\xi_2$};
    \draw[->, thick] (-1.4,-3) -- (-1.1,-2.35);
    \node at (-3.15, -2.15) {$\omega_{1..1} \omega_{2..2}$};
    \draw[->, thick] (-2.9,-1.9) -- (-2.1,-0.15);
    \node at (-3.25,2.75) {$\varphi_{1..2}^l$};
    \draw[->, thick] (-2.9,2.4) -- (-2.25,1.75);
    \node at (-1.25,3.5) {$\varphi_{1..2}^l \xi_2$};
    \draw[->, thick] (-1.65,3.2) -- (-1.25,2.25);
    \node at (1,3.25) {$\omega_{1..2}$};
    \draw[->, thick] (0.85,3) -- (1,1.5);
    \node at (2.25,2.75) {$\omega_{1..2} \xi_3$};
    \draw[->, thick] (1.8,2.5) -- (1.62,1.2);
    \end{tikzpicture}
    \caption{For the Borromean rings shown, we represent the various forms in Theorem \ref{thm:int-formula} by their Poincar\'e or Lefschetz duals.
    %Thus the $\xi_i$ correspond to link components and the $\omega_{i..i}$ to Seifert surfaces.  
    %The forms $\varphi_{1..2}^l \xi_2$, $\ \varphi_{1..2}^l$, and $\omega_{1..2}$ respectively correspond to the portions above the red disk of the blue component 2, its closed tubular neighborhood, and its Seifert surface.
    }
    \label{fig:borr-rings}
\end{figure}
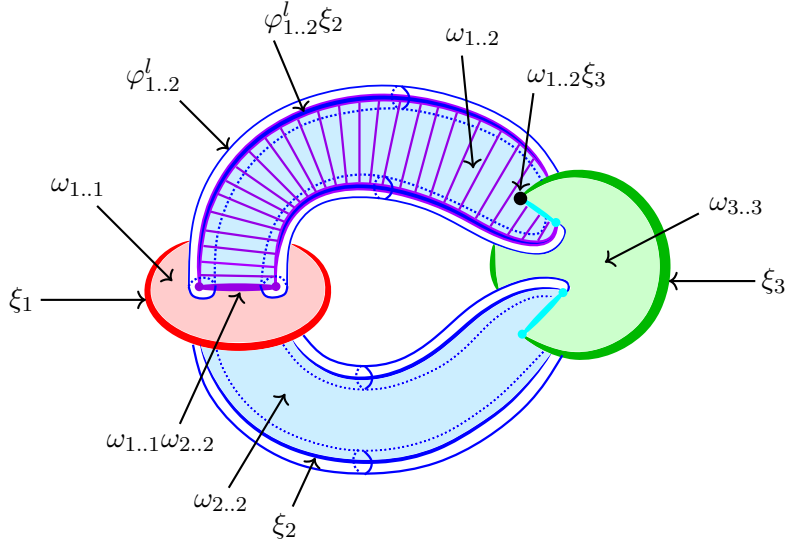

\begin{ex}
\label{ex:borr-rings}
For 3-component links, Theorem \ref{thm:int-formula} gives 
\[
\mu(1,2,3) = \int_{N_3} \left( \overline{\omega_{1..1}} \varphi_{2..3}^l \xi_3 + \overline{\omega_{1..2}}  \xi_3 \right),
\]
where we use that $\varphi_{3..3}^l=1$.
We apply it to  the Borromean rings shown in Figure \ref{fig:borr-rings}, where $m=3$ and each $p_i=1$, visualizing the various forms by their Poincar\'e or Lefschetz duals.
%, and we replace wedge product by intersection and $d$ by boundary.  
Thus the $\xi_i$ correspond to link components and the $\omega_{i..i}$ to Seifert surfaces.  
(This correspondence can be made more precise by using for example the geometric cochains of Friedman, Medina-Mardones, and Sinha \cite{FMMS}.)
%Thus $\xi_1$, $\xi_2$, and $\xi_3$ are the red, blue, and green link components, while  $\omega_{1..1}$, $\omega_{2..2}$, and $\omega_{3..3}$ are the red, blue, and green Seifert surfaces.
The first summand $\overline{\omega_{1..1}} \varphi_{2..3}^l \xi_3$ contributes nothing because, although $\varphi_{2..3}^l$ 
%(and $\varphi_{2..3}^r$) 
cannot be taken to be zero, $\omega_{1..1}$ and $\xi_3$ can be taken to have disjoint support, since component 3 does not intersect Seifert surface 1.
On the other hand, we may take $\varphi_{1..2}^r=0$, since there is a neighborhood of component 1 disjoint from Seifert surface 2.
We take $\varphi_{1..2}^l$ as the portion of the tubular neighborhood of the blue curve that lies above Seifert surface 1 (the red disk).  
Then the formula for $\omega_{1..2}$ simplifies to $d^{-1}(\overline{\omega_{1..1}} \omega_{2..2} + \varphi_{1..2}^l \xi_2)$.  
%The term $\overline{\omega_{1..1}} \omega_{2..2}$ corresponds to the intersection of the red and blue surfaces.   The term $\varphi_{1..2}^l \xi_2$ corresponds to the portion of the blue curve above the red disk.  
We take $\omega_{1..2}$ to correspond to the surface shaded by purple line segments.  Therefore $\overline{\omega_{1..2}}  \xi_3$ corresponds to the black intersection point between component 3 and Seifert surface 2.  Integrating an $m$-form dual to this point (and supported in a neighborhood of the green curve) gives $\mu(1,2,3)=\pm 1$.
\end{ex}

\subsection{Properties of Milnor invariants} \label{S:Milnor-props}
\begin{prop}
  In the notation of Theorem \ref{thm:int-formula}, the following quantities
  are equivalent definitions of the Milnor invariant $\mubar(1,\ldots,d)$:
  \begin{gather} \label{form-1}
    \int_{\partial N_d} \sum_{k=1}^{d-1}\overline{\omega_{1..k}}\omega_{k+1..d} \\
    \int_{N_d} \sum_{k=1}^{d-1}\overline{\omega_{1..k}}\ph_{k+1..d}^l\xi_d  \label{form-2} \\
    \int_{S^m} \sum_{k=1}^{d-1}\overline{\omega_{1..k}}\ph_{k+1..d}^l\xi_d  \label{form-3} \\
    \int_{f(S^{p_d})} \sum_{k=1}^{d-1}\overline{\omega_{1..k}}\ph_{k+1..d}^l. \label{form-4}
  \end{gather}
\end{prop}
\begin{proof}
  The equivalence of formulations \eqref{form-1} and \eqref{form-2} is given in the proof of
  Theorem \ref{thm:int-formula}.  Formulation \eqref{form-3} is equivalent to \eqref{form-2}
  simply because $\xi_d$ is zero outside $N_d$, and therefore so is the form
  being integrated.
  Finally, since $\sum_{k=1}^{d-1}\overline{\omega_{1..k}}\ph_{k+1..d}^l$ is a
  closed form and $\xi_d$ represents the Thom class of the bundle
  $N_d \to S^{p_d}$, \eqref{form-4}
  is equal to \eqref{form-2} by the Thom isomorphism.
\end{proof}
\begin{prop} \label{Milnor=period}
  The Milnor invariant $\mubar(1,\ldots,d)$ is a homotopy period of $f_d$ in
  $X=S^m \setminus (N_1 \cup \cdots \cup N_{d-1})$.
\end{prop}
\begin{proof}
  We continue to use the notation of Theorem \ref{thm:int-formula}.

  Consider the minimal DGA $\mathcal A$ generated by indecomposables $a_{i..j}$,
  where $1 \leq i \leq j \leq d-1$, $a_{i..i}$ has degree $q_i$ and
  $da_{i..i}=0$, and
  \[da_{i..j}=\sum_{k=i}^{j-1} \overline{a_{i..k}}a_{k+1..j}.\]
  By the results in \S\ref{S:wedges}, this is a subalgebra of the minimal model
  of $S^{p_1} \vee \cdots \vee S^{p_{d-1}}$.  Moreover, setting
  \[m(a_{i..j})=\omega_{i..j} \qquad\text{and}\qquad \ph(s\overline{a_{i..j}})=\ph^l_{i..j},\]
  we obtain that our formula \eqref{form-4} for $\mubar(1,\ldots,d)$ is
  exactly the formula for the homotopy period $\pi_{a_{1..d-1}}(f_d)$ computed
  using the construction \eqref{eq:Delta}.
\end{proof}
\begin{prop}
  The Milnor invariant $\mubar(1,\ldots,d)$ is well-defined, i.e., it does not depend
  on choices of primitives.
\end{prop}
\begin{proof}
  This follows from Proposition \ref{Milnor=period} and the fact that,
  according to Proposition \ref{props-of-periods}, homotopy periods are
  well-defined.
\end{proof}
\begin{prop} \label{shuffle}
  The Milnor invariants $\mubar(i_1,\ldots,i_d)$ satisfy shuffle relations.
  That is, given multiindices $I$ and $J$ whose lengths add up to $d-1$, we have
  \[\sum_\sigma (-1)^{\kappa(\sigma)}\mubar(\sigma(I,J),d)=0,\]
  where $\sigma$ runs over all shuffles of $I$ and $J$ and $\kappa(\sigma)$ is the
  Koszul sign
  \[\kappa(\sigma)=\sum_{i \in I, j \in J: \sigma(j)<\sigma(i)} (q_i+1)(q_j+1).\]
\end{prop}
\begin{proof}
  By Proposition \ref{Milnor=period}, this follows immediately from the
  corresponding facts for homotopy periods discussed in \S\ref{S:wedges}, starting with \cite[Proposition 3.21]{SW1}.
\end{proof}
\begin{prop} \label{cyclic}
  The Milnor invariants $\mubar(i_1,\ldots,i_d)$ satisfy a signed form of
  cyclic symmetry, namely,
  \[\mubar(1,\ldots,d)=(-1)^{m(q_1+1)}\mubar(2,\ldots,d-1,d,1).\]
\end{prop}
In the classical dimension, this proof was given by Mayer in his PhD thesis \cite[p.~37]{Mayer}.
\begin{proof}
  Let $\sigma_k$ be the cyclic permutation of $\{1, \dots, d\}$ that sends $1$ to $k$.
  By relabeling using $\sigma_k$, we can extend the definitions in Theorem \ref{thm:int-formula} to forms $\omega_{i..j}$ whenever $1\leq i, j \leq d$ and $i-j \not\equiv 0$ or $1 \operatorname{mod} d$.
  If we let
  \[\eta_k=\sum_{i=1}^d \overline{\omega_{k..i}}\omega_{i+1..(k-1 \operatorname{mod} d)},\]
  then according to \eqref{form-1},
  \[\mubar(\sigma_k(1,\ldots,d))=\int_{N_{(k-1 \operatorname{mod} d)}} \eta_k.\]
  Now
  \[\sum_{k=1}^d (-1)^{m \sum_{i=1}^{k-1} (q_i+1)}\eta_k=0\]
  because every term $\omega_{k..i}\omega_{i+1..(k-1 \operatorname{mod} d)}$ appears twice with opposite signs.  
  When we integrate this sum over $\partial N_1$, the terms $\eta_3,\ldots,\eta_d$
  evaluate to zero as in the proof of Theorem \ref{thm:int-formula}. Hence
  \[0=\int_{\partial N_1}(\eta_1+(-1)^{m(q_1+1)}\eta_2) = -\mubar(1,\ldots,d)+(-1)^{m(q_1+1)}\mubar(2,\ldots,d-1,d,1). \qedhere\]
\end{proof}
\begin{prop}\label{prop:massey-link-homotopy}
  The Milnor invariant $\mubar(1,\ldots,d)$ is a link homotopy invariant.
\end{prop}
The proof given here follows the outline suggested by Porter \cite[p.~42]{Porter}. %; in the next section we use the invariants defined by Koschorke in \cite{Kos2} to obtain an alternate explanation in certain special cases.
\begin{proof}
  %This follows from Propositions \ref{Milnor=period} and \ref{cyclic}.  
  Given a link homotopy, we can, by compactness, deform it to a concatenation of
  homotopies in which only one component moves at a time.  Now given a homotopy
  in which only the $d$th component moves, Proposition \ref{Milnor=period}
  and the homotopy invariance of homotopy periods imply that the Milnor invariant
  $\mubar(1,\ldots,d)$ is the same on both ends.  Cyclic symmetry from Proposition \ref{cyclic} 
  allows us to say the same for any component.
\end{proof}

\begin{prop} \label{prop:connectsum}
    Let $f_1,f_2:S^{p_1} \sqcup \cdots \sqcup S^{p_d} \to S^m$ be two links.  If $\mubar_{f_1}(1,\ldots,d)$ and $\mubar_{f_2}(1,\ldots,d)$ are well-defined, then any connected sum $g=f_1 \mathbin{\#} f_2$ satisfies
    \[\mubar_g(1,\ldots,d)=\mubar_{f_1}(1,\ldots,d)+\mubar_{f_2}(1,\ldots,d).\]
\end{prop}
The connected sum of two links is obtained by cutting out a ball inside each copy of $S^m$ which contains a ball inside each component of the link and identifying the boundaries.  This is not necessarily a well-defined operation: if some of the components have codimension 2, one may have multiple connected sums of the same two links that are topologically inequivalent.  For example, the Borromean rings may be obtained as a connected sum of a link consisting of a Hopf link and a third unlinked circle and its mirror image.  The same pair of links also have a trivial connected sum.  This does not contradict the Proposition \ref{prop:connectsum} since the triple linking number of these two is not well-defined.
\begin{proof}[Proof of Proposition \ref{prop:connectsum}]
    We use the functoriality of Massey products.  Define the following spaces:
    \begin{itemize}
        \item $X$, the link complement of $f_1 \mathbin{\#} f_2$; 
        \item $Y=X_1 \vee X_2$, the wedge of the link complements of $f_1$ and $f_2$; and
        \item $Z$, the subspace of $X$ obtained by removing, for each $i$, a small $D^{p_i} \times (-\epsi,\epsi)$ filling the ``neck'' connecting corresponding components of $f_1$ and $f_2$.
    \end{itemize}
    There are obvious maps
    \[X \xhookleftarrow{\iota} Z \xrightarrow{\pi} Y.\]
    By Alexander duality, $\pi$ induces isomorphisms on (co)homology below degree $m-1$ and a surjective map $H^{m-1}(Y) \to H^{m-1}(Z)$, and $\iota$ induces an injective map $H^*(X) \to H^*(Z)$.  Then $(\pi^*)^{-1}\iota^*(u_i^X)=u_i^{X_1}+u_i^{X_2}$, where the notation $u_i^A$ denotes the Alexander dual of $S^{p_i}$ in the space $A$.  So we see that if the Massey products $\langle u_1^{X_j},\ldots,u_d^{X_j} \rangle$ are unique, so is the Massey product $\langle u_1^X,\ldots,u_d^X \rangle$, and the summation formula holds.
\end{proof}

\subsection{Connection to Koschorke's link homotopy invariants}
\label{S:Massey-Milnor-Koschorke}
In this section, we recall the $\mu$-invariants introduced by Koschorke in \cite{Kos1,Kos2} and relate them to the Milnor invariants from Definition \ref{D:Milnor-invt}.  Although they are defined for somewhat different classes of maps, we show in Proposition \ref{P:Massey=Koschorke} below that the two notions coincide in the overlap between these classes.  Since Koschorke's invariants are invariant under link homotopy by definition, this provides an alternate, perhaps more natural proof of Proposition \ref{prop:massey-link-homotopy} in this restricted setting.  Koschorke also proved a result similar to Proposition \ref{Milnor=period} for links of codimension $\geq 3$.

%In this section, we establish a precise connection between the Milnor invariants $\mubar(1, \ldots, d)$ with distinct indices, defined via Massey products in Sections 4.1--4.3, and the higher-dimensional link homotopy invariants introduced by Koschorke in \cite{Kos1, Kos2}. Specifically, we demonstrate that for homotopy Brunnian links, our $\mubar$-invariants align with Koschorke’s $\mu$-invariants, which are derived from the Hopf homomorphism applied to a link class $\tilde{\kappa}(f)$. This alignment not only confirms the link homotopy invariance of our invariants (Proposition \ref{prop:massey-link-homotopy}) as a corollary but also ties them to the linking coefficient $\lambda(f)$ via Koschorke’s framework (c.f. Proposition \ref{Milnor=period}).

We begin by introducing the setting for Koschorke's work. Consider a spherical $d$-component \strong{link map} $ f: S^{p_1} \sqcup \cdots \sqcup S^{p_d} \to \mathbb{R}^m $, meaning $f$ is smooth and the component maps $ f_i: S^{p_i} \to \mathbb{R}^m $ have pairwise disjoint images. Unlike embeddings (assumed elsewhere in this paper due to our focus on thickness), link maps allow self-intersections within components.
%, aligning with Koschorke’s broader scope in \cite{Kos2}. 
A \strong{link homotopy} is a homotopy through link maps, preserving disjointness between components but permitting self-intersections within each.

Koschorke's $\mu$-invariants will be defined for \emph{$\kappa$-Brunnian} links, a definition which fits into a ladder of notions of Brunnian link:
\begin{itemize}
\item A $d$-component embedded link is \strong{Brunnian} if every $(d-1)$-component sublink is isotopic to the trivial link.
\item A $d$-component link map is \strong{homotopy Brunnian} if every $(d-1)$-component sublink is link homotopic to the trivial link.
\item A $d$-component link map
\[f=f_1 \sqcup \cdots \sqcup f_d:S^{p_1} \sqcup \cdots \sqcup S^{p_d} \to \RR^m\]
is \strong{$\kappa$-Brunnian} if for every $i$, the map
\[f_1 \times \cdots \times \widehat{f_i} \times \cdots \times f_d:S^{p_1} \times \cdots \times \widehat{S^{p_i}} \times \cdots \times S^{p_d} \to (\RR^m)^{d-1}\]
is nullhomotopic inside the configuration space ${C}_{d-1}(\RR^m)$ of $(d-1)$-tuples of distinct points in $\RR^m$.
\item A $d$-component embedded link is \strong{rational homotopy Brunnian} if the $\mubar$-invariants of Definition \ref{D:Milnor-invt} are well-defined and zero for all $(d-1)$-component sublinks.
\end{itemize}
Note that the homotopy Brunnian property is a priori stronger than the $\kappa$-Brunnian property: for a homotopy Brunnian link each stage of the nullhomotopy can be taken to be a product map.  Moreover, for embedded links we have
\[\xymatrix{\text{Brunnian up to link homotopy} \ar@{=>}[r] & \text{homotopy Brunnian} \ar@{=>}[r] \ar@{=>}[d] & \kappa\text{-Brunnian} \\ &\text{rational homotopy Brunnian.}
}\]
The arrows in the top row are equivalences at least in the classical case $m=3$ \cite[Cor.~6.2]{Kos2} and when all the components have codimension $\geq 3$ \cite[Cor.~6.4]{Kos2}.

% A link $f$ is {\em \textbf{Brunnian}} if every $(d-1)$-component sublink is link homotopic to the trivial link. Koschorke introduces a refined notion, the {\em \textbf{$\kappa$-Brunnian property}} \cite[Definition 4.1]{Kos2}, where $ f $ satisfies:
% \begin{itemize}
%     \item The restriction $f | S^{p_1} \times \cdots \times S^{p_{i-1}} \times \{*_i\} \times S^{p_{i+1}} \times \cdots \times S^{p_d} $ is nulhomotopic as a map into the configuration space ${C}_d(\mathbb{R}^m) := \{ (y_1, \ldots, y_d) \in (\mathbb{R}^m)^d \mid y_i \neq y_j \text{ for } i \neq j \} $ for each $ i = 1, \ldots, d $.
% \end{itemize}

Let $N = \sum_{i=1}^d p_i$. %For $N \leq d(m-2)$,
Then $f$ is $\kappa$-Brunnian if and only if the evaluation map
\[
\kappa(f): \prod_{i=1}^d S^{p_i} \to {C}_d(\RR^m), \quad \kappa(f)(t_1, \ldots, t_d) = (f_1(t_1), \ldots, f_d(t_d)),
\]
factors up to homotopy through the projection $\pi: \prod_{i=1}^d S^{p_i} \to S^N$ which collapses all proper faces to a point, yielding a map $\kappa(f): S^N \to {C}_d(\mathbb{R}^m)$.  By \cite[Proposition 2.2]{Kos2}, this then induces a well-defined homotopy class $\tilde{\kappa}(f) \in \tilde{\pi}_N(\bigvee_{d-1} S^{m-1})$, where $\tilde{\pi}_*(\bigvee_i S^{q_i}) \subseteq \pi_*(\bigvee_i S^{q_i})$ denotes the subgroup 
%(see \cite[eq.~(14)]{Kos2})
\[
\tilde{\pi}_\ast\left(\bigvee^{d-1}_{i=1} S^{q_i}\right):=\bigcap^{d-1}_{k=1} \ker \left(\pi_\ast\Bigl(\bigvee^{d-1}_{i=1} S^{q_i}\Bigr)\longrightarrow \pi_\ast\bigl(S^{q_1}\vee\ldots\vee \widehat{S^{q_k}}\vee\ldots\vee S^{q_{d-1}}\bigr)\right).
\]
In our notation from \S\ref{S:wedges}, Koschorke's $\mu$-invariants for a link map $f$ are defined as
\begin{equation}\label{eq:mu-gamma}
\mu_\sigma(f) := x_\sigma(\tilde{\kappa}(f))
\end{equation}
for each permutation $\sigma \in \Sigma_{d-1}$ which fixes $d-1$.
%, where $x_\sigma$ is shorthand for the functional $x_{(\sigma(1), \dots, \sigma(d-1))}$.
%, where $h_\gamma$ is a generalized Hopf homomorphism, originally considered in \cite{BoardmanSteer1967}, and given as a direct summand in the homomorphism 
%\begin{equation}\label{eq:h-h_gamma}
%h := \bigoplus_{\gamma \in \Sigma_{d-2}} h_\gamma: \tilde{\pi}_N\left(\bigvee_{d-1} S^{m-1}\right) \to \bigoplus_{(d-2)!} \pi_0^s \cong \bigoplus_{(d-2)!} %\end{equation}
As a refinement of our discussion in \S\ref{S:wedges}, when $N=(d-1)(m-1)-d+2$, the map $g \mapsto \{x_\sigma(g)\}$ defines an isomorphism \cite[Thm.~3.1]{Kos2}
\[\tilde\pi_N\left(\bigvee_{i=1}^{d-1} S^{m-1}\right) \to \ZZ^{(d-2)!},\]
with the $x_\sigma$ forming a dual basis to
\[\{\iota_\sigma=[\iota_{\sigma(1)},[\iota_{\sigma(2)},\cdots[\iota_{\sigma(d-2)},\iota_{\sigma(d-1)}] \cdots ]] : \sigma \in \Sigma_{d-1}\text{ fixes }d-1\}.\]
%Koschorke shows this in \cite[Theorem 3.1]{Kos2}.

To compare our invariants to Koschorke's, we use the inclusion $\RR^m \to S^m$ and the induced map on links in these ambient spaces.

\begin{prop}
\label{P:Massey=Koschorke}
    Let $f=f_1 \sqcup \cdots \sqcup f_d:S^{p_1} \sqcup \cdots \sqcup S^{p_d} \to \RR^m$ be an embedded link that is Brunnian up to link homotopy.  Then our Milnor invariants and those of Koschorke agree on $f$, i.e., 
    \[
    \mubar(\sigma(1),\ldots,\sigma(d-2),\sigma(d-1),d)=\mu_\sigma(f)
    \]
    for all permutations $\sigma \in \Sigma_{d-1}$ fixing $d-1$.  If either $m-p_i \geq 3$ for all $i=1,\dots, d$ or $m=3$ and $p_1=\dots=p_d=1$, the agreement also holds for $\kappa$-Brunnian links $f$.
\end{prop}

\begin{proof}
    For Brunnian embedded links $f$, Koschorke shows \cite[Thm.~6.1]{Kos2} that $\mu_\sigma(f)=\epsi x_\sigma(f_d)$, where $\epsi=\pm 1$ is a sign depending on $p_1,\ldots,p_d$ and $m$, and where we think of $f_d$ as a map
    \[f_d:S^{p_d} \to S^m \setminus (f_1(S^{p_1}) \cup \cdots \cup f_{d-1}(S^{p_{d-1}})) \simeq S^{q_1} \vee \cdots \vee S^{q_{d-1}}.\]
    (Note that the two instances of $x_\sigma$ in the defining equation \eqref{eq:mu-gamma} of $\mu_\sigma$ and in $x_\sigma(f_d)$ refer to functionals for wedges of spheres of different dimensions.)
    By Proposition \ref{Milnor=period} it follows that Koschorke's $\mu_\sigma(f)$ is equal (again up to sign) to our $\mubar(\sigma(1),\ldots,\sigma(d-1),d)$ for Brunnian links $f$ and, since both sides are invariant under link homotopy, for links $f$ that are Brunnian up to homotopy.\footnote{Since Koschorke's invariants $\mu_\sigma(f)$ are defined via homotopy periods of wedges of spheres, they retain some facets of a Lie coalgebra structure.  This structure, and tree-diagrammatic expressions for these invariants, were studied directly by the first two authors \cite{KKV}.}

    Recall \cite[Cor.~6.4]{Kos2} that when all components have codimension $\geq 3$, $\kappa$-Brunnian embedded links are always Brunnian up to homotopy.  Therefore, in that case, our invariants are defined and coincide with Koschorke's for all $\kappa$-Brunnian embedded links; indeed, they form a complete link homotopy invariant \cite[Prop.~8.4(c)]{Kos2}.  Similarly, in the case $m=3$, Koschorke's invariants \cite[Cor.~6.2]{Kos2} and ours \cite{Porter} both coincide (up to a sign depending only on $d$) with Milnor's classical invariants and therefore with each other.  
\end{proof}

In codimension $2$, $\kappa(f)$ and hence Koschorke's $\mu$-invariants are not complete link homotopy invariants: the Fenn--Rolfsen link \cite{FennR} is a link map $S^2 \sqcup S^2 \to S^4$ which (trivially) is homotopy Brunnian and (trivially) has trivial $\kappa$-invariant, yet it is not homotopically trivial.  On the other hand, it is also not homotopic to \emph{any} embedded link, and indeed for $m \geq 4$ all codimension-$2$ embedded links are homotopically trivial \cite{BPT:2025}.  If one allows both codimension-$2$ and higher-codimension components, it is to our knowledge still conceivable that there may be embedded links that are homotopy Brunnian or $\kappa$-Brunnian but not homotopic to a Brunnian link.

Nevertheless, we conjecture that the two types of invariants agree in all cases:

\begin{conj}
    For $\kappa$-Brunnian embedded links, Koschorke's $\mu_\sigma(f)$ coincides up to sign with our $\mubar(\sigma(1),\ldots,\sigma(d-1),d)$.
\end{conj}
Note that this would imply (by induction on $d$) that all $\kappa$-Brunnian embedded links are rational homotopy Brunnian.  Conversely, since Koschorke's invariants completely determine $\kappa(f)$ \cite[p.~311]{Kos2} and every set of invariants is represented by a Brunnian link, to prove the conjecture it would suffice to show that our $\mubar$-invariants are invariants of $\kappa(f)$.  

\section{Milnor invariants with repeated indices}
\label{S:repeated}

\subsection{Definition}
When taking a Massey product, there is no requirement that the cohomology classes whose product we are taking be distinct.  In particular, given a smooth embedding
\[f = (f_1, \dots, f_r):S^{p_1} \sqcup \cdots \sqcup S^{p_r} \to S^m,\]
with $u_i \in H^{q_i}(S^m \setminus f(S^{p_i}))$ indicating as before the Alexander dual of $f(S^{p_i})$, we can define the Massey product
\[\langle u_{\ell_1},\ldots,u_{\ell_d} \rangle\]
for any sequence of indices $\ell_i \in \{1,\ldots,r\}$.  We will show the following:
\begin{thm} \label{thm:repeated-exist}
    Suppose that $\ell_1 \neq \ell_d$ and that all Massey products corresponding to proper consecutive subsequences of $\ell_1,\ldots,\ell_d$ and their cyclic permutations are zero.  Then the Massey product $\langle u_{\ell_1},\ldots,u_{\ell_d} \rangle$ exists and is unique, specifically,
    \[
    \langle u_{\ell_1},\ldots,u_{\ell_d} \rangle=\mubar(\ell_1,\ldots,\ell_d)v_{\ell_1,\ell_d}
    \]
    for a unique integer $\mubar(\ell_1,\ldots,\ell_d)$.
\end{thm}
\noindent We define this integer to be the \emph{Milnor invariant} associated to the sequence $(\ell_1,\ldots,\ell_d)$.

 Our proof of Theorem \ref{thm:repeated-exist} roughly generalizes the cabling formula established by Milnor \cite[p. 297, Theorem 7]{Milnor2} in the classical case; see also \cite[Theorem 3.10]{Meilhan}.  That is, we will study these Massey products by relating them to those in which the indices are distinct.  To do so, we need suitable doubles of components, as guaranteed by the next lemma.  As before, for each $i=1,\dots, r$, let $N_i$ be the $i$th component of a tubular neighborhood of $f$.

\begin{lem}
    \label{lem:double-exists}
    There exists an embedding $g: S^{p_i} \to S^m$ such that 
    \begin{enumerate}[(i)]
    \item $g(S^{p_i})$ lies in $N_i \setminus f(S^{p_i})$ and is isotopic to $f_i$ in $N_i$; and
    \item  $g(S^{p_i})$ is rationally trivial in the complement of $f(S^{p_i})$, that is, the Milnor invariants $\mubar(i,r+1)$ and $\mubar(i,i,r+1)$ are both trivial, where $r+1$ is the label on $g$.
    \end{enumerate}
\end{lem}
\begin{proof}
    By a result of Massey \cite{Massey}, the normal bundle of $f(S^{p_i})$ is trivial as a spherical fibration and therefore has a non-vanishing section.
    %Since $f(S^{p_i})$ has trivial normal bundle, we can choose a section, which 
    Such a section allows us to define a map $g$ satisfying property (i).  
    From the correspondence between Massey products and homotopy periods, the two possibly nontrivial rational homotopy groups of $S^{q_i}$ are detected by Milnor invariants.  Namely, if $m=2p_i+1$, then $\mubar(i,r+1)$ measures the linking number of the two components, while if $m=\frac{3}{2}(p_i+1)$ and $q_i$ is even, $\mubar(i,i,r+1)$ measures the Hopf invariant of $g(S^{p_i})$ in the complement of $f(S^{p_i})$.  In either cae, we can modify $g$ on a small ball to trivialize the  invariant.  More specifically, we can take a connected sum with an embedding that represents the inverse of the value of the invariant, since such an embedding exists by Lemma \ref{lem:Whitehead}.
\end{proof}

\begin{proof}[Proof of Theorem \ref{thm:repeated-exist}] 
    %As in the previous section, we can replace every embedded sphere $f(S^{p_i})$ with a tubular neighborhood $N_i$.  The normal bundle of $f(S^{p_i})$ is trivial as a spherical fibration \cite{Massey}, and therefore has a non-vanishing section. 
    We choose multiple parallel spheres $S^{p_i}_j$ in $N_i$ satisfying the properties in Lemma \ref{lem:double-exists}, one for each instance of the index $i$ in $(\ell_1,\ldots,\ell_d)$.  Then with respect to the inclusion map
    \[S^m \setminus \bigcup_i N_i \hookrightarrow S^m \setminus \bigcup_{i,j} S^{p_i}_j,\]
    the Massey product $\langle u_{\ell_1}, \dots, u_{\ell_d}\rangle$ will be the pullback of a Massey product with distinct indices, which will establish the equality in the theorem statement.  It remains to show that, under the hypotheses and with our choice of parallel spheres, the Massey product we would like to pull back is defined and unique.

    We proceed by induction on the length $d$ of the Massey product.  For $d=2$, the indices are distinct by the hypotheses, so the result holds by Theorem \ref{thm:int-formula}.  Assume that the result holds for Massey products of length less than $d$.  We will prove it for length $d$ by adding doubles of components one at a time, decreasing $d-r$ by one with each double.  
    Repeating this step as many times as necessary, we will arrive at a situation in which $d-r=0$, that is, the indices $(\ell_1,\ldots,\ell_d)$ are distinct.  Provided the hypotheses are still satisfied after each step, we will deduce that the Milnor invariant exists and is unique by Theorem \ref{thm:int-formula}.  By the pullback argument above, this will prove the theorem.
    
    To this end, suppose that $f$ and $\ell=(\ell_1,\ldots,\ell_d)$ satisfy the hypotheses, and choose an index $i$ which occurs multiple times in $\ell_1,\ldots,\ell_d$, including at $\ell_k=i$.  Let $g$ be a double of $f_i$ as in Lemma \ref{lem:double-exists}.  Define the map $\tilde f:S^{p_1} \sqcup \cdots \sqcup S^{p_r} \sqcup S^{p_i} \to S^m$ by $\tilde{f}=(f_1, \dots, f_r, g)$.  Define a tuple $\tilde{\ell}=(\tilde\ell_1,\cdots,\tilde\ell_d)$ by replacing $\ell_k$ with $r+1$.  Then we claim that $\tilde{f}$ and $\tilde{\ell}$ still satisfy the hypotheses of the theorem, i.e., that the Massey product is zero for consecutive subsequences of $(\tilde\ell_1,\ldots,\tilde\ell_d)$ and their cyclic permutations.  
    
    % \noteblue{\rm To explain the inductive step more precisely, if $\ell_{i_1}$ is the first index repeating $k_1+1>1$ times, we make a substitution in the sequence $(\ell_1,\cdots,\ell_r)$ as follows 
    % \[
    % (\cdots,\ell_{i_1},\cdots, \ell_{i_1},\cdots,\ell_{i_1},\cdots) \longrightarrow (\cdots,\ell_{i_1},\cdots, d+1,\cdots,d+k_1,\cdots),
    % \]
    % for the $j$th index $\ell_{i_j}$, repeating $k_j+1>1$ times, we make a substitution 
    %  \begin{equation}\label{eq:repeated-to-non-repeated-seq}
    %  (\cdots,\ell_{i_j},\cdots, \ell_{i_j},\cdots,\ell_{i_j},\cdots) \longrightarrow (\cdots,\ell_{i_j},\cdots, d+K+1,\cdots,d+K+k_j,\cdots),
    %  \end{equation}
    %  where $K=k_1+\ldots+k_{j-1}$. Milnor in \cite[Theorem 7]{Milnor2} writes a more compact formula as follows 
    %  \[
    % \mubar(\ell_1, \ldots, \ell_r)(f) = \mubar(h(\ell_1), \ldots, h(\ell_r))(\tilde{f}),
    % \]
    % where $h(\,\cdot\,)$ is a function mapping the indices of the new components to the original components they parallel.
    %  }

    We split the verification of this claim into four cases:
    \begin{enumerate}[(a)]
    \item The subsequence does not contain $r+1$, or does not contain any instances of $i$.  In this case the Massey product calculation does not change through the inductive step.
    \item The subsequence is of the form $(i,\ldots,i,r+1)$ or a cyclic permutation thereof.  If the number of $i$'s is one or two, then the Massey product is zero by the construction of $g$.  If the number is greater than two, then the relevant homotopy period is always zero.
    %\end{enumerate}
    %Notice that this completes the proof of the claim for $r=2$.  
    %For the remaining cases, given a subsequence of indices of length $r'$, we will assume by induction that the claim holds for subsequences of length $<r'$, and therefore Theorem \ref{thm:repeated-exist} holds for sequences of indices of length $r'$.
    %\begin{enumerate}[(i),resume]
    \item The subsequence includes both $r+1$ and at least one instance of $i$, but not both occur at the ends of the subsequence.  Then the Massey product exists and is zero by the induction hypothesis on $d$.
    \item The subsequence is of the form $(i,\ldots,r+1)$ or $(r+1,\ldots,i)$.  In this case, we apply the claim to a cyclic permutation of the subsequence.  Induction on $d$ together with Proposition \ref{cyclic} imply that Milnor invariants with repeated indices satisfy (signed) cyclic symmetry, hence the relevant Massey product is again zero.
    \end{enumerate}
    This completes the proof of the claim and the theorem.
\end{proof}

\begin{rmk}
    When all components have codimension at least $3$, all lower-order Massey products are zero since they live in dimensions $<m-1$, making existence and uniqueness immediate.  In that case, the comparison to invariants with distinct indices in the first paragraph of the proof is still needed to show that the Massey product is a multiple of $v_{\ell_1,\ell_d}$.
\end{rmk}

% \begin{rmk}
% \noindent \noteblue{Ok?} In relation to Koschorke's invariants described in Section \ref{S:Massey-Milnor-Koschorke} and formula \warning{Not sure what you're going for here---we're discussing invariants with repeated indices so there's not a clear relationship to Koschorke} --  \noteblue{-- but they are pullbacks of ones with non--repeated indices, so Koschorke's construction should go through via doubling \ldots one important question (needed later): under the pull-back how do the iterated commutators behave? ... for the case of $\pi_1$ (relevant in the last section) we should be just substituting the repeated generator with nonrepeated ones\ldots but this does not feel right in higher dimensions} 

% \vdots 
% \end{rmk}

\subsection{Properties}
Using Theorem \ref{thm:repeated-exist}, we can immediately extend many properties from Section \ref{S:Massey} to Milnor invariants with repeats of indices allowed.  We summarize these here:
\begin{prop}[Properties of Milnor invariants with repeated indices] \ 
\label{prop:properties-repeated}
    \begin{enumerate}[(i)]
    \item \label{part:cyclic} Milnor invariants with repeated indices satisfy the cyclic symmetry of Proposition \ref{cyclic}, when defined.  That is, whenever $\ell_1 \neq \ell_d$ and $\ell_k \neq \ell_{k+1}$,
    \[\mubar(\ell_1,\ldots,\ell_d)=(-1)^{m(q_{\ell_1}+\cdots+q_{\ell_k}+k)}\mubar(\ell_{k+1},\cdots,\ell_d,\ell_1,\cdots,\ell_k).\]
    \item \label{part:period} If one index is not repeated, then the invariant can be expressed as a homotopy period of that component in the complement of the others.  In particular, the shuffle relations of Proposition \ref{shuffle} are satisfied.
    \item \label{part:concordance} Milnor invariants with possibly repeated indices are invariant under concordance. %, and hence up to topological isotopy (i.e.~homotopy through embeddings).
    \item \label{part:connectsum} Let $f_1,f_2:S^{p_1} \sqcup \cdots \sqcup S^{p_r} \to S^m$ be two links.  If $\mubar_{f_1}(\ell_1,\ldots,\ell_d)$ and $\mubar_{f_2}(\ell_1,\ldots,\ell_d)$ are well-defined, then any connected sum $g=f_1 \mathbin{\#} f_2$ satisfies
    \[\mubar_g(\ell_1,\ldots,\ell_d)=\mubar_{f_1}(\ell_1,\ldots,\ell_d)+\mubar_{f_2}(\ell_1,\ldots,\ell_d).\]
    In particular, if the codimensions of all the spheres are at least $3$, Milnor invariants define homomorphisms under the connected sum operation.
    \end{enumerate}
\end{prop}
Properties \ref{part:cyclic} and \ref{part:period} generalize those demonstrated by Milnor \cite{Milnor2} in the classical setting.  Property \ref{part:concordance} was shown in that setting by Casson \cite{Casson}, strengthening Milnor's result that his invariants are invariant with respect to (topological) isotopy; it can be derived from a result of Stallings \cite{Sta} which Casson evidently came up with on his own.
\begin{proof}
    Properties \ref{part:cyclic} and \ref{part:period} are immediate from Theorem \ref{thm:repeated-exist}.  Property \ref{part:connectsum} is proved in exactly the same way as Proposition \ref{prop:connectsum}.
    
    %, but \ref{part:isotopy} requires proof, since an isotopy in this sense can for example unknot any knotting in a component by pulling it tight \noteblue{(Robin: only in the $C^0$-topology, right?)}.  
    To prove property \ref{part:concordance}, consider a concordance
    \[F:(S^{p_1} \sqcup \cdots \sqcup S^{p_r}) \times [0,1] \to S^m \times [0,1] \]
    and abbreviate $S^{\mathbf{p}} = S^{p_1} \sqcup \cdots \sqcup S^{p_r}$.
    %We may assume that $F$ takes the form of a product on $[0, \epsi)$ and $(1-\epsi, 1]$ for some small $\epsi>0$.
    Extend $F$ via a trivial concordance to $(-\epsi,0]$ and $[1,1+\epsi)$ for some small $\epsi>0$. 
    We can think of $S^m \times (-\epsi,1+\epsi)$ as $S^{m+1}$ without the north and south poles.   
    Then by Alexander duality, for each $t \in [0,1]$, the inclusion 
    \[S^m \setminus F(S^{\mathbf{p}} \times \{t\}) \to 
    S^m \times (-\epsi,1+\epsi) \setminus 
    F(S^{\mathbf{p}} \times (-\epsi,1+\epsi)) \]
    induces an isomorphism on cohomology, since the complement of the right-hand side in $S^{m+1}$ is homotopy equivalent to the suspension of $S^\mathbf{p}$.  Massey product sets are functorial; in particular, the preimage of a unique Massey product under a map inducing an isomorphism in cohomology is also unique.  Therefore, the Massey products we are considering are independent of $t$.  
    %An isotopy $(x,t) \mapsto h(x,t)$ induces a concordance $F$ given by $F(x,t)=(h(x,t),t)$.
\end{proof}

The shuffle relations from Proposition \ref{prop:properties-repeated}\ref{part:period} force some Milnor invariants with repeated indices to be zero: for example, $\mubar(1,1,2)$ is $0$ whenever $S^{p_1}$ has even codimension, but it can be nontrivial when the codimension of $S^{p_1}$ is odd.  In fact, one readily sees that $\mubar(1,1,2)$ is the Hopf invariant of $S^{p_2}$ in the complement of $S^m \setminus S^{p_1} \simeq S^{q_1}$.

Proposition \ref{prop:properties-repeated}\ref{part:concordance} implies that Milnor invariants are also invariant under isotopy and ambient isotopy.  Indeed, an isotopy in the sense of homotopy through embeddings gives rise to a concordance, and an ambient isotopy gives rise to a homotopy through embeddings.  (By the isotopy extension theorem, an ambient isotopy is equivalent to a path of embeddings in the $C^k$-topology, $k\geq 1$, whereas a homotopy through embeddings is equivalent to a path in the $C^0$-topology.)  Concordance actually coincides with ambient isotopy in codimension at least 3 \cite{Hudson:1970}, but in general, it is only at least as coarse a relation as isotopy.  It is also at least as fine as link homotopy \cite{Giffen, Goldsmith, BPT:2025}.
%(Proposition \ref{prop:properties-repeated}\ref{part:concordance} holds for 

All Milnor invariants respect concordance
%, but some of them can distinguish links that are link homotopic but not concordant, such as the unlink and the Whitehead link.
but those with repeated indices are not in general link homotopy invariants.  (For example, $\mubar(1,1,2,2)$ detects the Whitehead link.)  That is because a link homotopy is no longer a link homotopy if one of the components is doubled.   On the other hand, some of them still carry some homotopy-invariant information.  For example, the homotopy class of $S^{p_2}$ in the complement of $S^{p_1}$ suspends to a link homotopy invariant lying in the stable homotopy group $\pi_{p_1+p_2-m+1}^s$, known as the $\alpha$-invariant \cite[Proposition 4.2]{MR}; see also \cite[\S5]{Ker} and \cite{HaK}.  
The $\alpha$-invariant is congruent mod $2$ to $\mubar(1,1,2)$
for $p_2=3, 7,$ or $15$, 
and for $p_2$ and $m$ such that $\mubar(1,1,2)$ is defined. 
%$m-p_1-1=2,4,$, or $8$,
(This follows from the fact that the Hopf map generates the appropriate stable homotopy group, and the Whitehead square, which has Hopf invariant 2, generates the kernel of the suspension map.)

\subsection{Example}
\label{S:mu112}
Consider two-component links $S^{p_1} \sqcup S^{p_2} \to S^m$ with $q_1$ even and dimensions satisfying $p_2=2q_1-1$, for example $p_1=p_2=4k-1$ and $m=6k$ for some $k$.  This is the requirement for the invariant $\mubar(1,1,2)$ to have a chance of being nonzero.  We show that it indeed takes on infinitely many values.

The easiest way to see this is as follows.  Consider the Borromean rings of dimensions $p_1$, $p_1$, and $p_2$ each at least $2$, which can be specified by explicit equations in $\RR^m$, such as
\begin{align*}
    L_1:&& \mathbf x &= 0 & \frac{\lvert \mathbf y \rvert^2}{\alpha^2}+\frac{\lvert \mathbf z \rvert^2}{\beta^2} &= 1 \\
    L_2:&& \mathbf y &= 0 & \frac{\lvert \mathbf z \rvert^2}{\alpha^2}+\frac{\lvert \mathbf x \rvert^2}{\beta^2} &= 1 \\
    L_3:&& \mathbf z &= 0 & \frac{\lvert \mathbf x \rvert^2}{\alpha^2}+\frac{\lvert \mathbf y \rvert^2}{\beta^2} &= 1
\end{align*}
for some fixed $\alpha>\beta>0$ and $\mathbf x$, $\mathbf y$, and $\mathbf z$ being vectors of dimensions $q_1$, $q_1$, and $q_2$.  Every pair of components is unlinked, but the Milnor triple linking number is $\pm 1$.

Now use a tube to surger together the two equidimensional components $L_1$ and $L_2$, creating a new component $L$.  If the codimensions of the components are at least $3$, a general position argument gives that
\[\mubar_{L \cup L_3}(1,1,2)=\sum_{i,j=1,2} \mubar_{L_1 \cup L_2 \cup L_3}(i,j,3)=\pm 2;\]
in other words $L_3$ has Hopf invariant $2$ in the complement of $L$.  (Indeed, the map $S^{p_2} \to S^{q_1}$ associated to $L \cup L_3$ is given by post-composing $[\iota_1, \iota_2]: S^{p_2} \to S^{q_1}\vee S^{q_1}$ by the fold map $S^{q_1}\vee S^{q_1} \to S^{q_1}$.  Alternatively, one could proceed as in Example \ref{ex:borr-rings} to calculate $\mubar_{L \cup L_3}(1,1,2)$.)

Similarly, links whose $\mubar(1,1,2)$ is exponential in the thickness can be constructed using the construction in \S\ref{S:exp} and surgering together two of the components.  %Finally, for $p_2=3$, $7$, or $15$, we can construct links with $\mubar(1,1,2)=1$ by taking the Hopf map and deforming it to an embedding; for example, if $p_1 \geq p_2$ it can be done by a graph construction due to Zeeman.

In the special case $p_1=p_2=:p$, there is also a nontrivial $\mubar(1,2,2)$.  Since $q:=m-p-1$ is assumed even, we must have $m=6k$, $p=4k-1$, and $q=2k$.  For the link $L \cup L_3$ just described, $\mubar(1,2,2)$ is zero.  (Indeed, $L$ is trivial in the complement of $L_3$, since it is the sum of the classes of $L_1$ and $L_2$, both of which are zero.)  See 
%\cite[Example 3.3]{Skopenkov:2008},
\cite[Lemma 3.4]{Skopenkov:2024} for further details and
\cite[\S6]{Haef4} 
%\warning{and Skopenkov?  I tried looking at his papers but they are chaotic and the references in the Manifold Atlas post don't match} 
for a generalization.  By taking connected sums of this link and the  link obtained by switching the labels on the two components, we get links with $(\mubar(1,1,2),\mubar(1,2,2))=(a,b)$ for any $a,b \in 2\ZZ$.  When $k \neq 1,2,4$, the Hopf invariant cannot be odd, so these are all possible values.

In the remaining cases of links $S^{4k-1} \sqcup S^{4k-1} \to S^{6k}$ with $k=1,2,4$, by combining work of Kervaire \cite[\S5]{Ker} with the EHP exact sequence
\[0 \to \pi_{2k}(S^{2k}) \xrightarrow{[\iota,\iota]} \pi_{4k-1}(S^{2k}) \xrightarrow{S} \pi_{4k}(S^{2k+1}) \to 0,\]
we see that $\mubar(1,1,2)$ and $\mubar(1,2,2)$ are still the same modulo 2.  It remains to find a link with $\mubar(1,1,2)=\mubar(1,2,2)=1$.  This is given by the following construction due to Zeeman \cite[\S10]{HaefLinks}.  Consider points in $S^{6k}$ as consisting of two vectors $(\mathbf x,\mathbf y)$ where $\mathbf x$ is $4k$-dimensional and $\mathbf y$ is $2k$-dimensional.  Then embed the two spheres via the maps
\[\mathbf x \mapsto (\mathbf x,\mathbf 0) \qquad\text{and}\qquad \mathbf x \mapsto \left(\frac{1}{\sqrt{2}}\mathbf x, \frac{1}{\sqrt{2}}h(\mathbf x)\right),\]
where $h$ is the Hopf map.  
(Similarly, for any $k$, there is a generating set consisting of the image of a generator of $\pi_{4k-1}(S^{2k})$ under this construction together with one of the two links obtained by tubing together the Borromean rings.  See work of the second author \cite[Theorem E, part (c)]{Koytcheff:2025} for a generalization of this generating set to families of long links.)
    
Haefliger also shows in \cite[\S10]{HaefLinks} that $\mubar(1,1,2)$ and $\mubar(1,2,2)$, together with the knotting invariants of the components, form a complete isotopy invariant of links up to torsion; see also \cite{Sko}.

\subsection{Rational isotopy classes in codimension at least 3}
The example worked out in \S\ref{S:mu112} is a special case of a more general fact.  Consider links $S^{p_1} \sqcup \cdots \sqcup S^{p_r} \to S^m$ with $p_i \leq m-3$.  They form an abelian group $L_{\mathbf p}^m=L_{(p_1,\ldots,p_r)}^m$ under connected sum, and Crowley, Ferry and Skopenkov \cite{CFS} give a thorough calculation of $L_{\mathbf p}^m \otimes \QQ$.  Their work together with that of Haefliger \cite{HaefLinks} implies that Milnor invariants, together with the Haefliger knotting invariants \cite{Haef} of the individual components, give a complete rational isotopy invariant, as we will show in Proposition \ref{prop:Q-iso-classes}.  (Compare this with Koschorke's \cite[Prop.~8.4(c)]{Kos2}, which shows that, with the same condition on dimensions, Milnor invariants with distinct indices give a complete \emph{integral} link homotopy invariant.)

We begin by reviewing the relevant information from those works.  First, the group structure means that isotopy classes of links can be written as a direct sum
\[L_{\mathbf p}^m \cong (L_{\mathbf p}^m)_U \oplus \bigoplus_{i=1}^r L_{p_i}^m \cong
\bigoplus_{\text{subsequences $\mathbf p'$ of }\mathbf p} (L_{\mathbf p'}^m)_B,\]
where $(L_{\mathbf p}^m)_U$ is the subgroup of links whose components are unknotted, 
each $L_{p_i}^m$ is a group of knots, 
and $(L_{\mathbf p}^m)_B$ is the subgroup of Brunnian links.  Haefliger gave a long exact sequence computing $(L_{\mathbf p}^m)_U$, whose rationalization splits \cite[Lemma 1.3]{CFS}  into short exact sequences
\begin{equation*}
%\label{eq:haef-SES}
0 \to (L_{\mathbf p}^m)_U \otimes \QQ \xrightarrow{\lambda \otimes \QQ} \Lambda_{\mathbf p}^{\mathbf q} \otimes \QQ \xrightarrow{w \otimes \QQ} \Pi_{\mathbf p}^{\mathbf q} \to 0,
\end{equation*}
where $\mathbf q=(q_1,\ldots,q_r)$,
\begin{align*}
    \Lambda_{\mathbf{p}}^{\mathbf{q}} &= \bigoplus_{i=1}^r
    \ker \left( \pi_{p_i}\left(\bigvee_{j=1}^r S^{q_j}\right) \to \pi_{p_i}(S^{q_i}) \right) 
    \cong \bigoplus_{i=1}^r \pi_{p_i}\left(\bigvee_{j=1}^r S^{q_j}; \, S^{q_i}\right), \text{ and}\\
    \Pi_{\mathbf{p}}^{\mathbf{q}} &=
\ker \left(\pi_{m-1}\left(\bigvee_{i=1}^r S^{q_i}\right) \to \bigoplus_{i=1}^r \pi_{m-1}(S^{q_i}) \right).
\end{align*} 
The map $\lambda$ is given by $(\lambda_1, \dots, \lambda_r)$, where each linking coefficient $\lambda_i$ is the equivalence class of a parallel copy of the $i$th component in the link complement, which 
is homotopy equivalent to  
% has the same homotopy groups in dimensions below $m-2$ as 
the wedge-sum $S^{q_1} \vee \dots \vee S^{q_r}$.  
The map $w$ is the sum $\sum_{i=1}^r w_i$ where $w_i(\alpha)$ is the Whitehead product $[\alpha, \iota_i]$ with the inclusion $\iota_i$ of the $i$th summand.  We use only the injectivity of $\lambda \otimes \QQ$ below. %\noteblue{Rafal: Say one sentence how this ties to Koschorke's $\lambda(f)$ \ldots }

\begin{prop}
\label{prop:Q-iso-classes}
    Milnor invariants define an injective homomorphism 
    \begin{align*}
    (L_{\mathbf p}^m)_U \otimes \QQ &\to \bigoplus_{I \in \mathcal I(\mathbf p)} \QQ
    \end{align*}
    given by $f \mapsto (\mubar_f(I))_{I \in \mathcal I(\mathbf p)}$, where $\mathcal I(\mathbf p)$ consists of all multiindices $I=(i_1, \dots, i_d)$, $d \geq 2$, such that equation \eqref{eq:dim-assumption} holds.
\end{prop}

\begin{proof}
    First, the reformulation \eqref{eq:dim-assumption-reformulation} of the assumed equality \eqref{eq:dim-assumption} shows that in codimension $\geq 3$, there are no lower-order invariants whose vanishing we need to check to ensure that the $\mu(I)$ are defined; therefore, by Proposition \ref{prop:properties-repeated}\ref{part:connectsum}, each $\mu(I)$ defines a homomorphism.  It also ensures that $\mathcal I(\mathbf p)$ is finite.  %To see that each $\mu(I)$ defines a homomorphism, view each invariant as the pullback under (possibly iterated) doubling of components of an invariant with distinct indices, as in the proof of Theorem \ref{thm:repeated-exist}. For two isotopy classes $[f]$ and $[g]$, we may take the connected sum $[\tilde f + \tilde g]$ of doubles $[\tilde f]$ and $[\tilde g]$ as a double of the connected sum $[f+g]$.  Since an invariant with distinct indices is a homotopy period (by Proposition \ref{Milnor=period}), we deduce that $\mu(I)$ is a homomorphism.  
    %(by Proposition \ref{props-of-periods}).
    %\noteblue{(I think this argument works, right?)}

    For injectivity, suppose $f \in (L_{\mathbf p}^m)_U \otimes \QQ$ satisfies $\mu_f(I)=0$ for all $I\in \mathcal I(\mathbf p)$.  Consider $\lambda_i(f)$ for any $i\in \{1, \dots, r\}$. 
    By the Hilton--Milnor theorem, 
    %it lies in 
    each direct summand of $\Lambda_{\mathbf{p}}^\mathbf{q}$ is 
    a subspace of the free graded Lie algebra on $r$ generators, where the $j$th generator corresponds to the inclusion $\iota_j$, and where the bracket corresponds to the Whitehead product.  
    Thus $\lambda_i(f)$ lies in the span of all iterated Whitehead products of the $\iota_j$ whose dimension is $p_i$.
    But the coefficient in $\lambda_i(f)$ of any such product $[\iota_{j_1}, [\iota_{j_2}, \dots [\iota_{j_{d-2}}, \iota_{j_{d-1}}] \dots]]$, $d\geq 2$, is equal to $\mubar(j_1, \dots, j_{d-1}, i)(f)$ by Proposition \ref{Whitehead/period}.
    Thus $\lambda_i(f)=0$, so $\lambda(f)=0$, and by the injectivity of $\lambda \otimes \QQ$, we obtain $f=0$, as desired.
\end{proof}

\subsection{Quantitative considerations}
We are able to prove estimates for Milnor invariants with repeated indices in terms of the thickness, as we are for those with distinct indices.  However, in some cases these estimates may not be sharp.  This is related to the quantitative aspect of doubling a component: can we double a component without decreasing the thickness too much?  In other words, can we find a section of the normal bundle of the component which does not ``wind too quickly'' around the zero section?

Naively, one can attempt to choose a section of the normal bundle of a $p$-sphere embedded in an $m$-sphere, skeleton by skeleton over a geometric triangulation of the embedded sphere.  Before getting to the $(m-p)$-skeleton, this process does not run into any obstructions, and one can always choose the most efficient possible representative.  At stage $m-p$, one runs into an obstruction cochain representing the Euler class of the normal bundle.  
As noted in the proof of Lemma \ref{lem:Whitehead}, this class is always cohomologically trivial, but if it is hard to trivialize (i.e., if the primitive has large $L^\infty$-norm), then every section of the normal bundle must have large winding on every simplex, and therefore must be thin.

We first point out that this cannot happen if $p<m/2$, since the obstruction is never reached:
\begin{prop} \label{prop:normal-easy}
    If $p<m/2$, then a $\tau$-thick embedding of $S^p$ in $S^m$ has a double in which both components are $c_{m,p}\tau$-thick, where $c_{m,p}>0$ is a constant.  In particular this holds for $p=1$ and $m=3$.
    \qed
\end{prop}
On the other hand, in the classical case $p=1, m=3$, this double may have large linking number with the original component, up to $\tau^{-4}$.  To ``unwind'' this linking number and produce a double which does not link with the original component, we may need to reduce the thickness:
\begin{prop} \label{prop:normal-1-3}
    When $p=1$ and $m=3$, a $\tau$-thick embedding of $S^p$ in $S^m$ has a double whose linking number with the original embedding is zero, in which the additional component is $c\tau^2$-thick, where $c>0$ is a constant.
    \qed
\end{prop}

We also consider the case $p=m-2$, which is the most relevant for the quantitative questions we consider.  When $p>1$, the issue of linking numbers does not arise, but we get a similar estimate:
\begin{prop} \label{prop:normal-codim2}
    If $p=m-2$, then a $\tau$-thick embedding $f:S^p \to S^m$ has a double in which the additional component is $c_{m,p}\tau^2$-thick, where $c_{m,p}>0$ is a constant.
\end{prop}
\begin{proof}
    As explained in the proof of Lemma \ref{lem:Whitehead}, in codimension 2, the Euler class is the only obstruction to trivializing the normal bundle,
    %since $S^1$ has trivial higher homotopy groups.   
    and it is trivial.

    Now we can fix a triangulation of $M=f(S^p)$ by approximating it using simplices of a triangulation $\mathcal T$ of $S^m$ as follows.  We fix a bilipschitz homeomorphism from $S^m$ to the boundary of a cube, subdivide the cube using a grid of side length $c_m\tau$, for sufficiently small $c_m>0$, and then triangulate the grid cubes.  Then there is an $L_m$-bilipschitz homeomorphism, again for a constant $L_m$, from $M$ to a (nearby) subcomplex $\Sigma_M$ of $\mathcal T$; see \cite{Whi,BKW}.

    Assume without loss of generality that $M$ is in general position with respect to this triangulation.  Then we can choose a simplicial representative $w \in C^2(\mathcal T)$ of the Thom class in $H^2(S^m)$: for each simplex, take the intersection number of a simplex with $M$ (which is always between $-1$ and $1$).  Restricting this to $\Sigma_M$ gives us a representative of the Euler class, namely, the obstruction to extending the $1$-skeleton of $\Sigma_M$ to a double of $M$.

    Now we choose a simplicial $1$-chain $a \in C^1(\mathcal T)$ whose coboundary is $w$.  By \cite[\S3]{CDMW}, we can choose this so that $\lVert a \rVert_\infty \leq C_m\tau^{-1}$.  (This is a kind of coisoperimetric inequality, as discussed further in \S\ref{S:coIP}.  It is easy to see via differential forms that one can choose a primitive of this size with real coefficients, but one has to do somewhat more work to get the result with integer coefficients.)  Restricting $a$ to $\Sigma_M$, we obtain the number of times we need to wind each edge of $\Sigma_M$ around $M$ in order to get a map on the $1$-skeleton which extends to a map on the $2$-skeleton (and therefore to all of $\Sigma_M$) which does not intersect $M$.  It follows that such an extension can be obtained with thickness $\sim \tau^2$.
\end{proof}
Note that once we have chosen one section, we can create more parallel components by interpolating between the first two.  Thus we can create an arbitrary number of components whose thickness has the same order, with only implicit constants depending on the number of components.

% \subsection{Examples}
% The polynomial construction from \S\ref{S:poly} clearly extends to the case of repeated invariants.  The $1$-dimensional component is never repeated, and so we can create large multiples of a Whitehead producrt in which the same term is repeated multiple times.

% The exponential construction is trickier since it's possible for every component to be repeated.  To handle this, we show the following relationship:
% \begin{prop}
%   Let $d \geq 3$, and consider a link
%   \[f:S^{p_1} \sqcup \cdots \sqcup S^{p_d} \to S^m.\]
%   Suppose that $p_{d-1}=p_d$.  We build a new $(d-1)$-component link $g$, equal
%   to $f$ on the first $d-2$ components, by letting $g(S^{p_{d-1}})$ be an
%   embedding of the connected sum of $S^{p_{d-1}}$ and $S^{p_d}$ such that the
%   symmetric difference of $f(S^{p_{d-1}}) \cup f(S^{p_{d-1}})$ and
%   $g(S^{p_{d-1}})$ is a nullhomotopic sphere in the complement of the rest
%   of the components.  Then for any $i_1,\ldots,i_r$,
%   \[\mubar_g(i_1,\ldots,i_r)=\sum (-1)^?\mubar_f(i_1',\ldots,i_r'),\]
%   where the sum is taken over all ways of assigning $i_j'=d-1$ or $d$ for all
%   $j$ for which $i_j=d-1$, and assuming that all the Milnor invariants on the
%   right side are well-defined.
% \end{prop}
% \begin{proof}
  
% \end{proof}

\section{Upper bounds on Milnor invariants using Massey products}
\label{S:upper-bounds}

Here we collect into one statement the parts of Theorems \ref{main} and \ref{main-repeated} that we will prove in this section:

\begin{thm} \label{thm:upper-bound}
  Suppose that $f:S^{p_1} \sqcup \cdots \sqcup S^{p_r} \to S^m$ is an embedding of thickness $\tau$ whose Milnor invariants indexed by proper subsequences of $(\ell_1,\ldots,\ell_d)$ are trivial.
  \begin{enumerate}[(a)]
  \item \label{case:exp} In all cases, $\mubar(\ell_1,\ldots,\ell_d) \leq \exp(C(m,d)\tau^{-m})$.
  \item \label{case:2} When $d=2$, $\mubar(\ell_1,\ell_2) \leq C(m)\tau^{-(m+1)}$.
  \item \label{case:poly} If the $\ell_j$ are distinct and one of the $p_{\ell_j}$ is $1$ (in which case the rest must be $m-2$), then
      \[\mubar(\ell_1,\ldots,\ell_d)\leq C(m,d)\tau^{-(m+1)(d-1)}.\]
  \item \label{case:poly-rep} If one of the $p_{\ell_j}$ is $1$ and not all the $\ell_j$ are distinct, then
      \[\mubar(\ell_1,\ldots,\ell_d)\leq C(m,d)\tau^{-2(m+1)(d-1)}.\]
  \item \label{case:Lip} If the $\ell_j$ are distinct and in addition $f$ is a locally $L$-bilipschitz map (with respect to the round radius $1$ metric on each sphere), then
      \[\mubar(\ell_1,\ldots,\ell_d) \leq C(m,d)\tau^{-\sum_{i=1}^d (q_i+1)}L^{(2m-5)(d-2)} = C(m,d)\tau^{-(m+2d-3)}L^{(2m-5)(d-2)}.\]
  \end{enumerate}
\end{thm}

Part \ref{case:2}, the linking number bound, is proved by Freedman and Krushkal using the Gauss formula, but can also be shown using norms of differential forms:
\begin{proof}[Proof of Theorem \ref{thm:upper-bound}\ref{case:2}]
    We assume $\ell_1=1$ and $\ell_2=2$ for ease of notation.  Note that $p_1+p_2=q_1+q_2=m-1$.  Then
    \[\mubar(1,2)=\operatorname{Lk}(S^{p_1},S^{p_2})=\int_{S^m} \xi_1 \wedge \omega_2,\]
    where $\xi_i \in \Omega^{q_i+1}(S^m)$ is a Poincar\'e dual to $f(S^{p_i})$ supported on the $\tau/2$-neighborhood $N_i$ of $f(S^{p_i})$, and $\omega_i$ is a primitive of $\xi_i$.  We can choose $\xi_i$ so that 
    \[\lVert\xi_i\rVert_\infty \leq C(m)\tau^{-(q_i+1)}\]
    for some constant $C(m)$, and therefore, by applying Lemma \ref{lem:genIP} below to the fixed manifold $M=S^m$, we can choose $\omega_i$ so that $\lVert\omega_i\rVert_\infty \leq C(m)\tau^{-(q_i+1)}$ as well.  It follows that
    \[\lvert\mubar(1,2)\rvert \leq C(m)\tau^{-(m+1)}. \qedhere\]
\end{proof}
The other statements will follow from the integral formulas for Milnor invariants in \S\ref{S:Massey} together with some (co)isoperimetric lemmas bounding the operator norms of primitives of forms in the tubular neighborhoods of link components.

\begin{rmk}\label{rmk:constants}
    Throughout this section, as in the first proof above, we frequently use notations like ``$C(a)$'' and ``$C(u,v)$'' to denote unspecified constants depending respectively on $a$ and on $u$ and $v$, which may change from one instance to the next.
\end{rmk}

\subsection{Coisoperimetric inequalities} \label{S:coIP}
Now we discuss coisoperimetric inequalities for differential forms and simplicial cochains.  
These inequalities will say that under some geometric hypothesis, a cochain or differential form $\omega$ has a primitive $\alpha$ whose norm is bounded by a constant, called a \emph{coisoperimetric constant}, times the norm of $\omega$.

\begin{lem}
\label{lem:matrix}
  Let $A$ be an $M \times N$ integer matrix such that
  \begin{itemize}
  \item all entries are $\pm 1$ or $0$ and
  \item every row has at most $p$ nonzero entries.
  \end{itemize}
  Let $b \in \mathbb R^m$ be a vector.  If the equation $Ax=b$ has a solution
  over $\mathbb R$, then it has a solution with
  \[\lVert x \rVert_\infty \leq \min\{m,n\}p^{\min\{m,n\}-1}\lVert b \rVert_\infty.\]
\end{lem}
\begin{proof}
  Let $r$ be the rank of $A$.  By reordering the bases for $\mathbb R^M$ and
  $\mathbb R^N$, we can assume that $A$ has a rank-$r$ invertible matrix $R$ in
  the upper left-hand corner.  Then $x=R^{-1}(b_1,\ldots,b_r) \times (0)^{N-r}$
  is a solution to $Ax=b$.  Now the entries of $R^{-1}$ are of the form
  $\det C_{ij}/\det R$ where $C_{ij}$ is a cofactor of $A$.  Since $C_{ij}$ is up to sign an
  $(r-1) \times (r-1)$ minor of $A$, and $A$ has at most $p$ nonzero entries per
  row, $|\det C_{ij}| \leq p^{r-1}$, by Hadamard's inequality.  Thus the entries of $x$ are at most
  $rp^{r-1}\lVert b \rVert_\infty$ in absolute value.
\end{proof}
\begin{cor} \label{cor:sIP}
  Let $X$ be a simplicial complex with $M$ $q$-simplices and $N$
  $(q-1)$-simplices.  Then every simplicial $q$-coboundary $b$ in $X$ is the
  coboundary $\delta c$ of some $(q-1)$-cochain $c$ satisfying
  \[
  %\pushQED{\qed}
  \lVert c \rVert_\infty \leq \min\{M,N\}(p+1)^{\min\{M,N\}-1}\lVert b \rVert_\infty.
  %\qedhere
  %\popQED
  \]
\end{cor}
\begin{proof}
    Apply Lemma \ref{lem:matrix} to the coboundary map, using that
    a $q$-simplex has $q+1$ simplices of dimension $q-1$ in its boundary.
\end{proof}

To show a continuous version of this fact, we introduce a lemma from work of the third author that helps transition between discrete and continuous coisoperimetric inequalities: 

\begin{lem}[{Second quantitative Poincar\'e lemma \cite[Lemma 2--4]{PCDF}}] \label{lem:2QP}
  For every $0<q \leq p$, there is a constant $C_{p,q}$ such that the following
  holds.  Let $\omega \in \Omega^q(\Delta^p)$ be a closed $q$-form, and let
  $\alpha_\partial \in \Omega^{q-1}(\partial\Delta^p)$ be a primitive for $\omega|_{\partial\Delta^p}$.  If $p=q$, we also require
  that the pair satisfies Stokes' theorem, that is,
  $\int_{\Delta^q} \omega=\int_{\partial\Delta^q} \alpha_\partial$.  Then there is
  a $(q-1)$-form $\alpha \in \Omega^{q-1}(\Delta^p)$ extending $\alpha_\partial$
  such that $d\alpha=\omega$ and $\lVert\alpha\rVert_\infty \leq
  C_{p,q}(\lVert\omega\rVert_\infty+\lVert\alpha_\partial\rVert_\infty)$.
  \qed
\end{lem}
Now we proceed with the continuous version of Corollary \ref{cor:sIP}:
%\noteblue{Robin: important to say ``thickness $\geq \tau$ rather than thickness $\tau$ below?}
\begin{lem} \label{lem:genIP}
  Let $M$ be an $p$-dimensional submanifold of $\RR^m$ or $S^m$, of thickness $\geq \tau$, perhaps with a boundary which also has thickness $\geq \tau$.  Then every exact $q$-form $\omega$ on $M$ 
  %is $d \alpha$ where 
  has a primitive $\alpha \in \Omega^{q-1}(M)$ such that 
  \[\lVert\alpha\rVert_\infty \leq C(p,m)\tau^{-(p-1)}\vol(M) \exp(C(p,m)\tau^{-p}\vol(M)) \lVert\omega\rVert_\infty.\]
\end{lem}
\begin{rmk}
    Rather than $\RR^m$ or $S^m$, one could use any fixed ambient manifold, although this will affect the constants.
\end{rmk}
\begin{proof}[Proof of Lemma \ref{lem:genIP}]
  By rescaling, we can assume that $\tau=1$.  This multiplies $\lVert\alpha\rVert_\infty$ by $\tau^{q-1}$, $\lVert\omega\rVert_\infty$ by $\tau^q$, and the volume of $M$ by $\tau^{-p}$, accounting for the power of $\tau$ in the inequality.

%  Now by a theorem of Boissonnat, Dyer and Ghosh \cite[Theorem 3]{BDG} (see also \cite[\S4]{BDGW} for a slight correction, as well as work of Cheeger, M\"uller, and Schrader \cite{CMS} prefiguring these ideas),
  Now $M$ has a triangulation $T$ whose simplices are $C(p,m)$-bilipschitz to the standard linear simplex.  When the ambient manifold is $\RR^m$, this can be obtained by approximating $M$ using a grid, as remarked already by Whitney; see \cite{Whi,BKW}.  In particular, the number of simplices of $T$ is proportional to the volume of $M$.  For the sphere, one can first impose the standard cubical structure, and then subdivide.

  Let $w \in C^q(T;\RR)$ be the simplicial cochain obtained by integrating
  $\omega$ over the simplices of $T$; then $\lVert w \rVert_\infty \leq C(p,m)\lVert\omega\rVert_\infty$.  By the de Rham theorem, $w$ is a coboundary,
  and by Corollary \ref{cor:sIP}, $w=\delta a$ where
  \[\lVert a \rVert_\infty \leq C(p,m)\vol(M)\exp(C(p,m)\vol (M))\lVert w \rVert_\infty.\]
  Finally, we use $a$ and $\omega$ to construct the form $\alpha$ skeleton by
  skeleton, with the help of Lemma \ref{lem:2QP}.  We start with the
  $(q-1)$-simplices; on each $(q-1)$-simplex $\sigma$ we set
  \[\alpha|_\sigma=a(\sigma)\ph d\vol,\]
  where $\ph$ is a fixed bump function which is zero on $\partial\sigma$.  Then
  we extend $\alpha$ to each higher skeleton inductively; by Lemma
  \ref{lem:2QP}, at each step the operator norm increases by at most a constant
  $C_{p,q}$.
\end{proof}
Now we show that in the lowest and highest dimensions, we can do better: the
coisoperimetric constants are linear in the volume, rather than exponential.
\begin{lem} \label{lem:bottomIP}
  Let $M$ be a Riemannian manifold of diameter $D$.  Then every exact $1$-form
  $\omega$ on $M$ has a primitive $f:M \to \mathbb R$ such that 
  $\lVert f \rVert_\infty \leq \frac{D}{2}\lVert \omega \rVert_\infty$.
\end{lem}
\begin{proof}
  Since $\omega$ is exact, we can construct the function $f$ by setting it to be zero at some
  point and integrating along paths from that point.  Then the difference between
  the maximal and minimal value of $f$ is at most $D\lVert\omega\rVert_\infty$;
  by adding a constant we can make the desired inequality hold.
\end{proof}
\begin{lem} \label{lem:topIP}
  Let $M$ be an $p$-dimensional submanifold of $\RR^m$ of thickness $\tau$.  Then
  every exact $p$-form $\omega$ on $M$ 
  %is $d\alpha$ where
  has a primitive 
  $\alpha \in \Omega^{p-1}(M)$ such that 
  \[\lVert \alpha \rVert_\infty \leq c_{p,m}\vol(M)\tau^{-(p-1)}\lVert\omega\rVert_\infty,\]
where $c_{n,m}$ is a constant.
\end{lem}
\begin{rmk}
    The value of $c_{p,m}$ cannot be less than $(2\vol\mathbb S^{p-1})^{-1}$, where $\mathbb S^{p-1}$ is the round unit $(p-1)$-sphere; this is demonstrated by taking $M$ to be an arbitrarily long sausage of radius $\tau$.  One could hope to show that this is sharp, if a sharp isoperimetric inequality for small volumes of the type discussed in \cite{JoMo,Druet} holds in the largest possible range.  Alternatively, one can use Almgren's sharp isoperimetric inequality in $\RR^m$ \cite{Alm} to fill small $(p-1)$-cycles in $M$ and then project the resulting filling to $M$; this strategy at least gives a value of $c_{p,m}$ which only depends on $p$, but a full proof of this does not seem relevant to this paper.
\end{rmk}
\begin{proof}[Proof of Lemma \ref{lem:topIP}]
    The proof proceeds similarly to that of Lemma \ref{lem:genIP}.  Fix a triangulation $T$, and let $w \in C^m(T;\RR)$ be the simplicial cochain obtained by integrating $\omega$.  The fact that $w$ is a coboundary is equivalent to its total sum over all oriented $m$-simplices being zero.  So we can write $w$ as a sum of cochains supported on two simplices,
    \[w=\sum_i w_i=\sum_i c_i(\chi_{\sigma_i}-\chi_{\sigma'_i}),\]
    with $\sum_i c_i=\frac{1}{2}\lVert w \rVert_1$.  Now each $w_i$ has a primitive $a_i$ with $\lVert a_i \rVert_\infty=c_i$, supported on a path from $\sigma_i$ to $\sigma'_i$ in the dual graph of $T$.  Then $w=\delta a$, where $a=\sum_i a_i$, and therefore
    \[\lVert a \rVert_\infty \leq \sum_i c_i \leq \frac{1}{2}\lVert w \rVert_1 \leq \frac{1}{2}\lvert T \rvert\lVert w \rVert_\infty \leq C(p,m)\vol(M)\lVert\omega\rVert_\infty.\qedhere\]
\end{proof}

Finally, we show that coisoperimetric constants transfer from a thickly embedded manifold to its tubular neighborhood:
\begin{lem} \label{lem:tubeIP}
  Suppose that $M$ is an $p$-dimensional submanifold of $\RR^m$ of thickness
  $\tau$, and suppose that every exact $\omega \in \Omega^q(M)$ has a primitive
  $\alpha \in \Omega^{q-1}(M)$ such that
  \[\lVert\alpha\rVert_\infty \leq A\lVert\omega\rVert_\infty.\]
  Let $N$ be the embedded normal bundle of $M$ of radius $\tau/4$.  Then every
  exact $\omega \in \Omega^q(N)$ has a primitive $\alpha \in \Omega^{q-1}(M)$ such that
  \[\lVert\alpha\rVert_\infty \leq (2^{q-1}A+1)\lVert\omega\rVert_\infty.\]
\end{lem}
\begin{proof}
  First note that this property is scale-free, so we can assume $\tau=1$.  We
  start by letting $\alpha_0 \in \Omega^{q-1}(M)$ be a primitive of $\omega|_M$
  with the desired bound.  Our goal is to extend $\alpha_0$ to a primitive of
  $\omega$ on all of $N$ without increasing the operator norm too much.  We do
  this by taking
  \[\alpha(x)=\pi^*\alpha_0(x)+\int_{\pi(x)}^x \omega,\]
  where $\pi:N \to M$ is the orthogonal projection and $\int_{\pi(x)}^x$ denotes
  the fiberwise integral along the straight line from $\pi(x)$ to $x$.  Since
  $\pi$ is $2$-Lipschitz and the distance from $\pi(x)$ to $x$ is at most $1/4$,
  this form satisfies the desired bound.
\end{proof}

\subsection{Proof of Theorem \ref{thm:upper-bound}}
For each of the statements, we prove bounds by estimating operator norms of forms.  For the estimate in Theorem \ref{thm:upper-bound}\ref{case:exp}, which holds very generally, we use the standard Massey product formulation of Milnor invariants from Proposition \ref{simpler}, which generalizes to any sequence of indices.  For the other estimates, we use the forms constructed in Theorem \ref{thm:int-formula}.
We start with the exponential general bound.

\begin{proof}[{Proof of Theorem \ref{thm:upper-bound}\ref{case:exp}}]
  We first consider $\mubar(1, \dots, d)$.
  Choose $N_i$ to be tubular neighborhoods of $f_i(S^{p_i})$ of radius $\tau/4$.  This makes the complement of any collection of the $N_i$ into a codimension-$0$ submanifold which is $\tau/4$-thick with $\tau/4$-thick boundary, allowing us to apply Lemma \ref{lem:genIP}.  Applying it inductively, we will obtain bounds on the operator norms of the forms 
  \[\omega_{i..j} \in \Omega^*(S^m \setminus (N_{\ell_i} \cup \cdots \cup N_{\ell_j})).\]

  First note that we may choose $\xi_i$ Poincar\'e dual to $f_{\ell_i}(S^{p_{\ell_i}})$ so that 
  \[\lVert\xi_i\rVert_\infty\leq C(p_i,m)\tau^{-(q_i+1)}.\]
  Since the $\omega_{i..i}$ are obtained by integrating these in $S^m$, we get an analogous bound for $\lVert\omega_{i..i}\rVert_\infty$.  Now by Lemma \ref{lem:genIP},
  \[\lVert\omega_{i..j}\rVert_\infty \leq C(m,m)\tau^{-(m-1)}\exp(C(m,m)\tau^{-m})\left(\sum_{k=i}^{j-1} \lVert\omega_{i..k}\rVert_\infty\lVert\omega_{k+1..j}\rVert_\infty\right).\]
  By induction on $d$, we see that
  \[\lVert\omega_{1..d}\rVert_\infty \leq \operatorname{poly}(\tau^{-1})\exp((d-1)C(m,m)\tau^{-m}).\]
  Finally, we can choose $\theta_{1,d}$ so that $\lVert\theta_{1,d}\rVert_\infty=\tau/8$, so by Proposition \ref{simpler}, the bound on $\lVert\omega_{1..d}\rVert_\infty$ also applies to $\mubar(1,\ldots,d)$.  
  Ignoring the polynomial (which can be deleted after increasing the constant in the exponent), we obtain the desired bound on $\mubar(1,\ldots,d)$.  
  Finally, recalling that Proposition \ref{simpler} extends to arbitrary sequences $(\ell_1, \dots, \ell_d)$, we see that a completely analogous argument gives the same bound on $\mubar(\ell_1, \dots, \ell_d)$.
\end{proof}

Now we give a general outline using the integral formula of Theorem \ref{thm:int-formula} that allows us to plug in the various coisoperimetric estimates from \S\ref{S:coIP} to obtain the various results for distinct indices.  Throughout, we assume that the $N_i$ are tubular neighborhoods of $f_i(S^{p_i})$ of radius $\tau/4$.  To simplify notation, we will also assume that we are computing $\bar \mu(1,\ldots,d)$, since a modification of this argument will anyway be needed for repeated indices.

Let $B(i..j)$ be a bound on $\lVert\omega_{i..j}\rVert_\infty$ and let $B^{l/r}(i..j)$ be a bound on $\lVert\ph^{l/r}_{i..j}\rVert_\infty$.  We will inductively
produce estimates on these.  As above, we can take $B(i..i)=C(n,m)\tau^{-(q_i+1)}$.  We also know $\ph^{l/r}_{i..i}=1$, so $B^{l/r}(i..i)=1$.

Now let $A_{m,p,q}$ be the coisoperimetric constant for exact $q$-forms in an embedded $p$-manifold in $S^m$, subject to any geometric assumptions we may be making.  Then we have
\begin{align*}
    B^l(i..j) &\leq A_{m,p_j,q^l(i..j)}\cdot \sum_{k=i}^{j-1} B(i..k)B^l(k+1..j) \\
    B^r(i..j) &\leq A_{m,p_i,q^r(i..j)}\cdot \sum_{k=i}^{j-1} B^r(i..k)B(k+1..j),
\end{align*}
where $q^{l/r}(i..j)$ is given by
\[q^{l/r}(i..j)+q_{j/i}+1=\sum_{k=i}^j q_k-(j-i-2).\]
Similarly, since $\omega_{i..j}$ is obtained by integrating over the whole $m$-sphere,
\[B(i..j) \leq C(m)\cdot\left(\sum_{k=i}^{j-1} B(i..k)B(k+1..j)+\sum_{k=i}^j B^l(i..k)\tau^{-(q_k+1)}B^r(k..j)\right).\]

Let $A$ be the maximum value of $A_{m,p_j,q^l(i..j)}$ or $A_{m,p_i,q^r(i..j)}$ over the possible values of $i \leq j$.  Then by induction on $j-i$ we have
\begin{align*}
  B^l(i..j)\tau^{-(q_j+1)} &\leq C(m,d)A^{j-i}\tau^{-\sum_{k=i}^j (q_k+1)} \\
  B^r(i..j)\tau^{-(q_i+1)} &\leq C(m,d)A^{j-i}\tau^{-\sum_{k=i}^j (q_k+1)} \\
  B(i..j) &\leq C(m,d)A^{j-i}\tau^{-\sum_{k=i}^j (q_k+1)},
\end{align*}
and finally the integrand in \eqref{eqn:Massey} has operator norm bounded by
\begin{equation}
    \label{eq:bound-on-integrand}
    C(m)A^{d-2}\tau^{-\sum_{k=1}^d (q_k+1)} = C(m,d)A^{d-2}\tau^{-(m+2d-3)}.
\end{equation}
Since we integrate it over a region of constant volume, this bound is also the final bound on the Milnor invariant.

Now we consider what the values of $A$ may be in different cases.  In the most general case, by Lemmas \ref{lem:genIP} and \ref{lem:tubeIP}, 
\[A_{m,p,q} \leq C(p,m)\tau^{-(p-1)}\vol(M)\exp(C(p,m)\tau^{-p}\vol(M)).\]
Since $\vol(M) \leq \vol(S^m)/\tau^{m-p} \sim \tau^{-(m-p)}$, we can rewrite
this as
\begin{equation}
\label{eq:Ampq-bound}
    A_{m,p,q} \leq C(p,m)\tau^{-(m-1)}\exp(C(p,m)\tau^{-m}).
\end{equation}
%Therefore in this case the bound on the Milnor invariant is 
Substituting this quantity for $A$ in formula \eqref{eq:bound-on-integrand}, we obtain
\[\mubar(1,\ldots,d) \leq \operatorname{poly}(\tau^{-1})\exp((d-2)C(p,m)\tau^{-m}),\]
which is roughly the same bound on this Milnor invariant as in part \ref{case:exp}, albeit this time only for distinct indices.
We will prove parts \ref{case:poly}, \ref{case:poly-rep}, and \ref{case:Lip} by refining the estimate \eqref{eq:Ampq-bound}.

\begin{proof}[{Proof of Theorem \ref{thm:upper-bound}\ref{case:poly} and \ref{case:poly-rep}}]
Let $p_1=1$ and $p_2=\cdots=p_d=m-2$, and first suppose that we are still computing $\mubar(1,\ldots,d)$.  Then we get
\begin{align*}
  q^r(1..j)=q^l(1..1) &= 1 \\
  q^l(1..j) &= m-2 & j &\neq 1 \\
  q^{l/r}(i..j,k) &= 1 & i &\neq 1.
\end{align*}
In every case, $A_{m,p,q}$ is controlled by Lemma \ref{lem:bottomIP} or by Lemmas \ref{lem:topIP} and \ref{lem:tubeIP}.  The diameter of each $f_i(S^{p_i})$ is bounded by $C(m)\tau^{-(m-1)}$, since a geodesic in the embedded manifold is an embedded curve of thickness $\tau$, and its neighborhood has volume $\leq 1$.  Thus the constant of Lemma \ref{lem:bottomIP} is at most $C(m)\tau^{-(m-1)}$.  Similarly, the volume of an $(m-2)$-dimensional embedded manifold of thickness $\tau$ is at most $C(m)\tau^{-2}$, and so the constant of Lemma \ref{lem:topIP} is at most $C(m)\tau^{-(m-1)}$ as well.  This gives
\begin{equation}
\label{eq:mu-1-d-bound}
\mubar(1,\ldots,d) \leq C(m,d)\tau^{-((d-2)(m-1)+m+2d-3)} = C(m,d)\tau^{-(m+1)(d-1)}.    
\end{equation}

%Now suppose that some indices are repeated, and suppose $m=3$.  Then by Proposition \ref{prop:normal-easy}, we can reduce this to the case of non-repeated indices while only decreasing the thickness by a multiplicative constant.

Now suppose that we are estimating $\mubar(\ell_1, \dots, \ell_d)$ with repeats of indices allowed.  Notice that the dimension of the repeated components can only be $m-2$.  Then by Proposition \ref{prop:normal-1-3} (when $m=3$) or Proposition \ref{prop:normal-codim2} (otherwise), we can reduce this calculation to the case of non-repeated indices, but where the components corresponding to indices which were originally repeated may be at most $\sim \tau^2$-thick.  
Replacing $\tau$ by $\tau^2$ in the bound in formula \eqref{eq:mu-1-d-bound}, we get
\[\mubar(\ell_1,\ldots, \ell_d) \leq C(m,d)\tau^{-2(m+1)(d-1)}.\qedhere\]
\end{proof}

\begin{proof}[{Proof of Theorem \ref{thm:upper-bound}\ref{case:Lip}}]
Assume that each map $f_i$ is a locally $L$-bilipschitz map from a round sphere of radius $1$.  Then we can take primitives of $q$-forms in $f_i(S^{p_i})$ by pulling them back to $S^{p_i}$ (which multiplies the operator norm by at most $L^q$), taking the primitive there (which multiplies the operator norm by at most a constant), and then pulling the resulting $(q-1)$-form back again to $f_i(S^{p_i})$ (which multiplies the operator norm by at most $L^{q-1}$).  Then we have an estimate 
\[A_{m,p_i,q} \leq C(m,p_i,q)L^{2q-1} \leq C(m,p_i)L^{2p_i-1}.\]
Now rather than simply taking the largest possible value of this quantity to the power of $d-2$, we do a slightly more refined analysis.  By analyzing the integral formula \eqref{eqn:Massey} we see that in every product of $A_{m,p_i,q}$ that occurs in our bound, no index $i$ repeats.  Thus the product of $A_{m,p,q}$ resulting from any such chain is bounded by
\[C(m,d)L^{\sum_{i=1}^d (2p_i-1)}=C(m,d)L^{2((m-2)(d-1)+1)-d}.\]
In fact, since we are always taking products of at most $d-2$ factors (cf.~formula \eqref{eq:bound-on-integrand}), two of the $p_i$ will always be missing from the sum in the exponent.  The sum of the smallest two $p_i$ is always at least $m-1$, so we get that
\begin{align*}
  \mubar(1,\ldots,d) &\leq C(m,d)\tau^{m+2d-3}L^{2((m-2)(d-1)+1)-d-(2(m-1)-2)} \\
                      &= C(m,d)\tau^{m+2d-3}L^{(2m-5)(d-2)}.\qedhere
\end{align*}
\end{proof}

\section{Application to Freedman--Krushkal examples} \label{S:FK}

Freedman and Krushkal \cite{FK} studied the following question: given an $n$-complex which embeds in $\RR^{2n}$, how complicated can its simplest embedding be?  They proved an upper bound: for every $n>2$, every such complex with $N$ simplices embeds in the unit ball of $\RR^{2n}$ with thickness $O(\exp(-N^{4+\epsi}))$, with constants depending on $n$.  (On the other hand, for $n=2$, it seems likely that the complexity can be at least superexponential and perhaps cannot be bounded by any recursive function; for analogous results see the work of Boris Lishak, e.g.~\cite{Lishak-Nabutovsky:2017}.)  They also established a lower bound: for every $n \geq 2$, they construct a sequence of complexes $K_l$ with $O(l)$ vertices and at most $C(n)$ simplices adjacent to each vertex such that every embedding of $K_l$ in the unit ball has complexity at least $c(n)^{-l}$.  This lower bound is obtained by considering the linking number between two spherical subcomplexes of $K_l$; in any embedding, this turns out to be at least $2^l$.

In \cite[\S5]{FK}, they constructed additional $2$-complexes for which, in any embedding in $\RR^4$, certain spherical subcomplexes must have exponentially large $q$th-order linking coefficients, and they asked whether embeddings of these must also be exponentially thin.  We will use the polynomial regimes of Theorems \ref{main} and \ref{main-repeated} to confirm this conjecture.

%In this section, we apply our results on Milnor invariants of thick links, particularly Theorem \ref{main} (ii), to resolve a question posed by Freedman and Krushkal in \cite{FK}, Section 5, concerning the thickness of certain 2-complexes embedded in $\mathbb{R}^4$. Their work explores the geometric complexity of embeddings of $n$-complexes in $\mathbb{R}^{2n}$, demonstrating in \cite[Section 4]{FK}, by a specific telescoping construction that, such embeddings can require exponentially small thickness due to high linking numbers. They further extend this to 2-complexes in $\mathbb{R}^4$ where linking numbers vanish but higher-order invariants, such as Massey products, are non-trivial, and ask whether an exponential thickness bound persists. We confirm this conjecture by connecting their constructions to our framework of thick spherical links and Milnor invariants. 

	\begin{figure}[ht]
		\centering
		\includegraphics[width=0.4\linewidth]{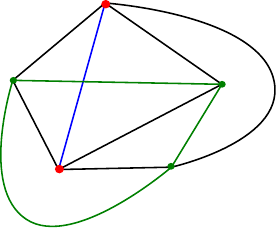}
		\caption{The graph $K^1_0$, immersed in $\RR^2$ with one intersection point between nonadjacent edges.  The 1-simplex $T$ is in blue, the 0-sphere $S_1$ consists of the two red vertices, and the 1-sphere $S_2$ is the green cycle.} \label{fig:k5}
	\end{figure}

\subsection{Exponentially thin $n$-complexes in $\mathbb{R}^{2n}$}
As a warmup, we summarize Freedman and Krushkal's construction of the complexes $K_l$ and their proof that they have large embedding complexity \cite[\S4]{FK}.

\subsubsection*{Base complex $K_0^n$} Denote the $(2n+2)$-simplex with vertices $v_0,\ldots,v_{2n+2}$ by $\Delta^{2n+2}=[v_0,\ldots,v_{2n+2}]$.  Define $K_0^n$ to be the $n$-skeleton of $\Delta^{2n+2}$ with one $n$-simplex $T = [v_0, \ldots, v_n]$ removed:
    \[
    K_0^n = \left(\Delta^{2n+2}\right)^{(n)} \setminus T.
    \]
    This complex contains an $(n-1)$-sphere $S_1 = \partial T$ and an $n$-sphere $S_2 = \partial[v_{n+1}, \ldots, v_{2n+2}]$; see Figure \ref{fig:k5} for the case $n=1$. By \cite[Lemma 6]{FKT}
    %\cite[Proposition 4.1]{FK} 
    (a result dating back to van Kampen \cite{vK}) for any embedding $i:K_0^n \to \mathbb{R}^{2n}$, the linking number $\link(i(S_1), i(S_2))$ is odd.\footnote{Conversely, any odd linking number can be achieved \cite{KaSk}.}

\subsubsection*{Mapping telescope $K_l$} We obtain $K_l$ by attaching to $K^n_0$ the $l$-fold mapping telescope 
\begin{equation}\label{eq:K_l} 
    K_l := \left(\underbrace{S^{n-1} \xrightarrow{\times 2} S^{n-1} \xrightarrow{\times 2} \cdots \xrightarrow{\times 2} S^{n-1} \xrightarrow{\times 2}}_{l\text{ times}} S_1 \subset K^0_n\right)
\end{equation}
so that the leftmost copy of $S^{n-1}$ represents $2^l$ times the homotopy class of $S_1$.  What follows is a more detailed account of this construction.

Fix a triangulation $\Sigma$ of $S^{n-1}$ and a simplicial map $g:\Sigma \to \partial\Delta^n$ of degree $2$.  Fix also a triangulation $\overline\Sigma$ of $S^{n-1} \times [0,1]$ which restricts to $\partial\Delta^n$ at time $0$ (i.e.~on $S^{n-1} \times \{0\}$) and $\Sigma$ at time $1$.  Then
\[M=\overline\Sigma \sqcup \partial\Delta^n/(x,1) \sim g(x)\]
is the mapping cylinder of $g$, with subcomplexes $M^0$ at time $0$ and $M^1$ at time $1$ both isomorphic to $\partial\Delta^n$.

Denote $l$ copies of $M$ by $M_1,\ldots,M_l$, and let
\[K_l=K_0^n \sqcup \bigsqcup_{i=1}^l M_i/S_1 \sim M^1_1, M^0_i \sim M^1_{i+1} \text{ for }1 \leq i \leq l-1.\]
Then $K_l$ is a simplicial complex with $O(l)$ simplices which deformation retracts to $K_0^n$.

\subsubsection*{Embedding thickness of $K_l$} %Note that $K_l$ embeds into the unit ball in $\mathbb{R}^{2n}$ via general position ($i: K_l \to B_1^{2n}$).
It is not hard to see that $K_l$ embeds in $\RR^{2n}$.  On the other hand, $[M^0_l] \in \pi_{n-1}(K_l \setminus S_2)$ is homotopic to $2^l[S_1]$, and so in any such embedding, the linking number of $M^0_l$ and $S_2$ is congruent to $2^l$ mod $2^{l+1}$.

Let $i:K_l \to B$ be an embedding of $K_l$ in the unit ball $B \subset \RR^{2n}$.  Then Freedman and Krushkal use the Gauss linking integral to show that the thickness $\tau$ of $M^0_l \sqcup S_2 \subset K_l$ satisfies
\[C(2n)\tau^{-(2n+1)} \geq \left\lvert \link(i(M^0_l), i(S_2)) \right\rvert \geq 2^l.\]
The first inequality is essentially Theorem
\ref{thm:upper-bound}\ref{case:2}.
%\ref{main}\ref{intro-case:link}, which can also be shown using \noteblue{(norms of?)} differential forms.
%Write 
%\[f:S^{n-1} \sqcup S^n  \longmapsto B,\quad f = i\bigl|_{M^0_l\sqcup S_2}.\]
%Then we obtain the estimate
%\[\left\lvert \link(i(M^0_l), i(S_2)) \right\rvert = \lvert\mubar(1,2)(f)\rvert \leq C(2n)\tau^{-(2n+1)}.\]
Thus, the thickness of $i$ satisfies the upper bound
\begin{equation}\label{eq:tau-K_L-upper}
    \operatorname{thickness}(i) \leq \tau \leq C(n) \cdot 2^{-l / (2n+1)}
\end{equation}
and is therefore exponentially small as $l \to \infty$.

%%%%%%%%%%%%%%%%%%%%%%%%%%%%%%%%%%
\subsection{Exponentially thin $2$-complexes in $\mathbb{R}^{4}$ via higher-order invariants}
In \cite[\S5]{FK}, Freedman and Krushkal construct 2-complexes $\bar{K}_{q,l}$ in $\mathbb{R}^4$ whose embeddings are constrained not by linking numbers between subcomplexes, but by higher-order linking invariants. We detail this construction, including an alternative version with $q+1$ copies of $K_0^2$, and show that these still have exponentially small thickness.

\subsubsection*{Base complex $\bar{K}_0$}
Take two copies of $K_0^2$ (the 2-skeleton of $\Delta^6$ minus a 2-simplex).  In these two copies, we distinguish the submanifolds $C'$ and $C''$ (boundary circles of the missing 2-simplices) and $S'$ and $S''$ (boundary 2-spheres of the 3-simplices on the four vertices not in $C'$ and $ C''$ respectively).  We identify a vertex in $C'$ with a vertex in $C''$ to get a complex
\[\bar{K}_0 = K_0^2 \vee K_0^2,\]
and denote the common vertex by $v$.

By van Kampen's obstruction 
\cite[Lemma 6]{FKT}
%\cite[Proposition 4.1]{FK}, 
for any embedding $i:\bar K_0 \to \RR^4$, the linking numbers $\link(i(C'), i(S'))$ and $\link(i(C''), i(S''))$ are odd.  On the other hand, the remaining pairs are unlinked:
\begin{lem} \label{lem:no-link}
    The linking numbers $\link(i(C'), i(S''))$ and $\link(i(C''), i(S'))$ are zero.
\end{lem}
\begin{proof}
    Let $w$ be a vertex of $S'$.  The join of $w$ and $C'$ is a subcomplex of $\bar K_0$, so $i(C')$ is isotopic to a circle contained in a ball arbitrarily close to $i(w)$.  Since $i(S'')$ is at a definite distance from $i(w)$, $i(S'')$ and $i(C')$ cannot be linked.  The other case is analogous.
\end{proof}

\subsubsection*{Commutators $\bar{K}_q$} 
%Define simplicial loops $x=C'$, $y=C''$ based at $v$.  
Let $x$ and $y$ be loops based at $v$ corresponding to $C'$ and $C''$ respectively.
We attach $q$ additional loops $w_1, \ldots, w_q$ to $\bar K_0$ at $v$ and then attach $q$ 2-cells $D_1, \ldots, D_q$ to $\bar K_0$ using the gluing maps
\[\partial D_1=w_1[x, y]^{-1}, \quad \partial D_2=w_2[x, w_1]^{-1}, \quad \ldots, \quad \partial D_q=w_q[x, w_{q-1}]^{-1},\]
to yield a complex $\bar K_q$.  This CW-complex can (up to homotopy) be triangulated with $O(q)$ simplices and 
\begin{equation}\label{eq:w_q}
    w_q = [x, [x, \cdots [x, y] \cdots ]]
\end{equation}
has word length $3 \cdot 2^q-2$ in the free group $F = \langle x, y \rangle$.

\subsubsection*{Alternative construction with $q+1$ copies of $K_0^2$} A variant uses $q+1$ copies of $K_0^2$, each with a boundary circle $C_j$ and a $2$-sphere $S_j$ ($j = 1, \ldots, q+1$), identified at a common vertex $v$ contained in each of the $C_j$.  As before, in any embedding of this wedge sum, $S_i$ and $C_j$ have odd linking number if $i=j$ and are unlinked otherwise.

%Define loops $x_j=C_j$ based at $v$.  
Let $x_j$ be a loop based at $v$ corresponding to $C_j$.
Attach $q$ additional loops $w'_1,\ldots,w'_q$ based at $v$, as well as $2$-cells $D'_1,\ldots,D'_q$ with boundary maps
\[\partial D_1'=w'_1[x_q, x_{q+1}]^{-1}, \quad \partial D_2'=w'_2[x_{q-1}, w'_1]^{-1}, \quad \ldots, \quad \partial D_q'=w'_q[x_1, w'_{q-1}]^{-1}\]
to yield a complex $\bar K_q'$.  This CW-complex can likewise be triangulated with $O(q)$ simplices.  Moreover,
\begin{equation}\label{eq:w_q'}
    w_q' = [x_1, [x_2, [x_3, \ldots, [x_q, x_{q+1}]]]]
\end{equation}
is a commutator in $q+1$ distinct generators in the free group $F = \langle x_1, \ldots, x_{q+1} \rangle$ (that is, it induces a nontrivial element of the Milnor group of $F$) and again has word length $3 \cdot 2^q-2$.
    
\subsubsection*{Mapping telescope $\bar{K}_{q,l}$}
As with $K_l$, we obtain $\bar K_{q,l}$ by attaching to $\bar K_q$ the $l$-fold mapping telescope
\begin{equation}\label{eq:bar-K_q,l} 
    \bar K_{q,l} := \left(\underbrace{S^1 \xrightarrow{\times 2} S^1 \xrightarrow{\times 2} \cdots \xrightarrow{\times 2} S^1 \xrightarrow{\times 2}}_{l\text{ times}} w_q \subset \bar K_q\right)
\end{equation}
so that the leftmost copy of $S^1$ (which we denote by $C_{q,l}$) represents $2^l w_q \in F$.

Note that $\bar K_{q,l}$ deformation retracts to $\bar K_0$; moreover,
\[C_{q,l} \simeq [x, [x, \cdots [x, y] \cdots ]]^{2^l} \in F=\pi_1(\bar K_{q,l} \setminus (S_1 \cup S_2)).\]
It follows that for any embedding $i:\bar K_{q,l} \to \RR^4$, $i(C_{q,l})$ represents the conjugacy class of an element
\[C_{q,l} \simeq [\alpha_1, [\alpha_1, \cdots [\alpha_1, \alpha_2] \cdots ]]^{2^l} \in \pi_1(\RR^4 \setminus (i(S_1) \cup i(S_2))),\]
where the homology class of $\alpha_j$ is Alexander dual to $r_j[S_j]$ for some odd $r_j$.

Likewise, we obtain $\bar K_{q,l}'$ by attaching an $l$-fold mapping telescope along $w_q'$ to $\bar K_q'$.  Denoting the leftmost copy of $S^1$ by $C_{q,l}$, we get that for any embedding $i:\bar K_{q,l}' \to \RR^4$, $i(C_{q,l})$ represents the conjugacy class of an element
\[C_{q,l} \simeq [\alpha_1, [\alpha_2, \cdots [\alpha_q, \alpha_{q+1}] \cdots ]]^{2^l} \in \pi_1\left(\RR^4 \setminus \bigcup_{j=1}^{q+1} i(S_j)\right),\]
where the homology class of $\alpha_j$ is Alexander dual to $r_j[S_j]$ for some odd $r_j$.

\subsubsection*{Embedding thickness}
Now we prove Theorem \ref{thm-C}, which we restate here:
\begin{thm}
    Any embedding of the $2$-complex $\bar{K}_{q,l}$ or $\bar{K}_{q,l}'$ has thickness at most $c^{-l}$, where the constant $c>1$ depends on $q$.
\end{thm}
\begin{proof}
    We first analyze the structure of the group $\Gamma=\pi_1\left(\RR^4 \setminus \bigcup_j i(S_j)\right)$.  By Proposition \ref{prop:almost-free}, the lower central series quotients $\Gamma/\Gamma_k$ are free $k$-step nilpotent groups generated by the Alexander duals of the $i(S_j)$.  In both cases, the homomorphism $i_*:F \to \Gamma$ induces an injection on $H_1({-},\QQ)$, and therefore the $\alpha_j$ likewise generate a free $k$-step nilpotent subgroup in each of the lower central series quotients.

    In particular, $i(w_q)$ (respectively, $i(w_q')$) induces a nontrivial element in $\Gamma_{q+1}/\Gamma_{q+2}$.  By Proposition \ref{prop:period=Magnus}, this pairs nontrivially with a homotopy period $\pi_{x_{i_1,\ldots,i_r}}$; therefore we have
    \[\pi_{x_{i_1,\ldots,i_r}}(i(C_{q,l}))=c \cdot 2^l\]
    for some constant $c$.  By Proposition \ref{Milnor=period}, adapted to repeating indices, this homotopy period is equal to the Milnor invariant $\mubar(q+2,i_1,\ldots,i_r)$ of the link
    \[i(S_1) \cup \cdots \cup i(S_{q+1}) \cup i(C_{q,l}).\]
    Finally, by Theorem \ref{main-repeated}\ref{intro-case:poly-rep}, the thickness $\tau$ of this link satisfies
    \[c \cdot 2^l \leq C(q)\tau^{-10(q+1)},\]
    which in turn gives us the desired exponential estimate for the thickness of the whole embedded complex:
    \[\operatorname{thickness}\bigl(\bar{K}_{q,l}\bigr)\leq \tau \leq  2^{-l / 10(q+1)} C(q).\]
    The same argument gives the result for $\bar K_{q,l}'$.
\end{proof}

\bibliographystyle{amsplain}
\bibliography{thickness-no-desk}
\end{document}